\documentclass[11pt]{amsart}

\usepackage{hyperref} 
\usepackage[dvips]{graphicx}
\usepackage[T1]{fontenc}
\usepackage{mathptmx}
\usepackage{color}
\usepackage[titletoc]{appendix}

\DeclareMathOperator{\trace}{trace}
\DeclareMathOperator{\Ad}{Ad}

\newcommand{\bbar}{\begin{pmatrix}}
\newcommand{\ebar}{\end{pmatrix}}

\newcommand{\bdm}{\begin{displaymath}}
\newcommand{\edm}{\end{displaymath}}
\newcommand{\beq}{\begin{equation}}
\newcommand{\beqa}{\begin{eqnarray}}
\newcommand{\beqas}{\begin{eqnarray*}}
\newcommand{\eeq}{\end{equation}}
\newcommand{\eeqa}{\end{eqnarray}}
\newcommand{\eeqas}{\end{eqnarray*}}
\newcommand{\dd}{\textup{d}}

\newcommand{\C}{{\mathbb C}}

\newcommand{\real}{{\mathbb R}}
\newcommand{\SSS}{{\mathbb S}}

\newcommand{\Z}{{\mathbb Z}}

   \newtheorem{theorem}{Theorem}[section]
   \newtheorem{proposition}[theorem]{Proposition}

   \newtheorem{lemma}[theorem]{Lemma}

 \theoremstyle{remark}
   \newtheorem{example}[theorem]{Example}

\numberwithin{equation}{section}

\begin{document}

\title{Spherical Surfaces}
\author{David Brander}
\address{Department of Applied Mathematics and Computer Science\\ Matematiktorvet, Building 303 B\\
Technical University of Denmark\\
DK-2800 Kgs. Lyngby\\ Denmark}
\email{dbra@dtu.dk}

\keywords{Differential geometry, integrable systems, loop groups, spherical surfaces, constant Gauss curvature, singularities, Cauchy problem}
\subjclass[2000]{Primary 53A05, 53C43; Secondary 53C42, 57R45}

\begin{abstract}
We study surfaces of constant positive Gauss curvature  in Euclidean 3-space via the harmonicity of the Gauss map.  Using the loop group representation, we solve the regular and the singular geometric Cauchy problems for these surfaces, and use these solutions to compute several new examples.  We  give the criteria on the geometric Cauchy data for the generic singularities, as well as for the cuspidal beaks and cuspidal butterfly singularities. We consider the bifurcations of generic one parameter families of spherical fronts and
provide evidence that suggests that these are the cuspidal beaks, cuspidal butterfly and one
other singularity.  We also give the loop group
potentials for spherical surfaces with finite order rotational symmetries and for surfaces with embedded
isolated singularities.
\end{abstract}
\maketitle

\section{Introduction}
\subsection*{Motivation and goals}
A \emph{spherical surface} is defined to be an immersed surface in Euclidean $3$-space, with positive constant 
induced Gauss curvature $K>0$.  Since the theory of such surfaces is
essentially the same for any positive constant $K$, we will use this term for the case $K=1$. 
It is well known  that the Gauss map of such an immersion is harmonic with respect to the
metric induced by the second fundamental form; 
 conversely, harmonic maps from a Riemann surface into
$\SSS^2$ correspond to spherical surfaces (with singularities).     The existence of the
holomorphic \emph{Hopf} quadratic differential, which vanishes precisely at umbilics, implies that the
only \emph{complete} spherical surface is necessarily a round sphere; thus any non-trivial global theory 
of spherical surfaces leads inevitably to the consideration of singularities, either as natural boundaries
for smooth surfaces or as part of a generalized surface.  There is a well-developed theory of harmonic maps into $\SSS^2$ that can be used to generate solutions (with natural singularities), and we will therefore work with these generalized spherical surfaces,  and call them \emph{spherical frontals}.

The purpose of this work is two-fold: firstly, although spherical surfaces are a classical
topic in differential geometry, there appears to be a dearth of concrete examples and
visualizations of them in the literature,
especially when compared with constant \emph{negative} curvature surfaces and flat surfaces.
The best known classical spherical surfaces are probably the surfaces of revolution, helical surfaces, Enneper's surfaces of constant positive curvature and the Sievert-Enneper surface.
Physical models of some of these can, at the time of writing, be found at 
\href{http://modellsammlung.uni-goettingen.de/}{http://modellsammlung.uni-goettingen.de/}.
A survey on the classical examples, with references, is the section by H.~Reckziegel in \cite{fischer}. 
Classically, the explanation for the lack of visualizations would have been the nonlinear nature of the problem; but
recently methods have been available to compute the solutions using implementations
of loop group theory, as has been done with non-minimal constant mean curvature (CMC) surfaces 
\cite{DorPW, KMS, kss2004, hellerschmitt}.
Spherical surfaces are obtained as parallel surfaces to CMC surfaces; however the geometric
relationship is not particularly intuitive (Figure \ref{figure0}). For constructing interesting examples it is 
more appropriate to study spherical surfaces directly.   

The other motivation for this project is to provide the ``elliptic complement'' to recent work \cite{dbms1, singps} on the geometric Cauchy problem
and its applications to the study of singularities of constant \emph{negative} 
curvature surfaces. The geometric
Cauchy problem (GCP) for a class of surfaces is, broadly speaking, to provide \emph{geometric} data, along a given curve in the ambient space, that is sufficient to define a unique solution surface of the geometry in question -- preferably with a means of constructing the solution. This problem has been solved for CMC surfaces in
\cite{bjorling}, where, as with Bj\"orling's classical problem, a prescribed curve and surface normal
along the curve are sufficient to construct the solution.  This does not solve the GCP for spherical surfaces, however
the underlying idea (see Section \ref{bjorlingsection} below) can be adapted for this case, as indeed 
it can to any other geometric problem characterized by a harmonic Gauss map.   
 G\'alvez, Hauswirth and Mira \cite{ghm} have treated the GCP for spherical surfaces, but not
from a loop group point of view.  They use the GCP to study
\emph{isolated} singularities, meaning cone points and branch points.
They focus on isolated singularities because they regard the surface as a point set in $\real^3$, 
rather than an immersion from a manifold into $\real^3$. From that point of view, 
they consider surfaces with isolated singularities to be the closest to complete smooth surfaces.

\subsection*{Organization and results of this article}
In Section \ref{background} we sketch the necessary definitions and results on loop group 
methods for spherical surfaces. To illustrate the method we give, in 
Section \ref{symmetriessection} the potentials for spherical surfaces with finite order rotational
symmetries and compute a number of examples (Figures \ref{figureB}, \ref{figureA}).
We also discuss regularity and the criteria for branch points.

		\begin{figure}[ht]
\centering
$
\begin{array}{ccc}
\includegraphics[height=33mm]{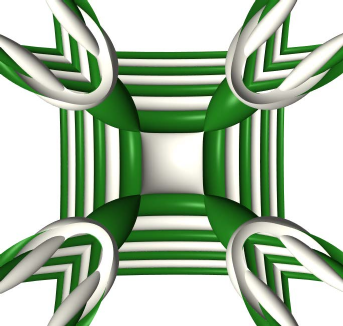}  \,\, &  \,\,
\includegraphics[height=33mm]{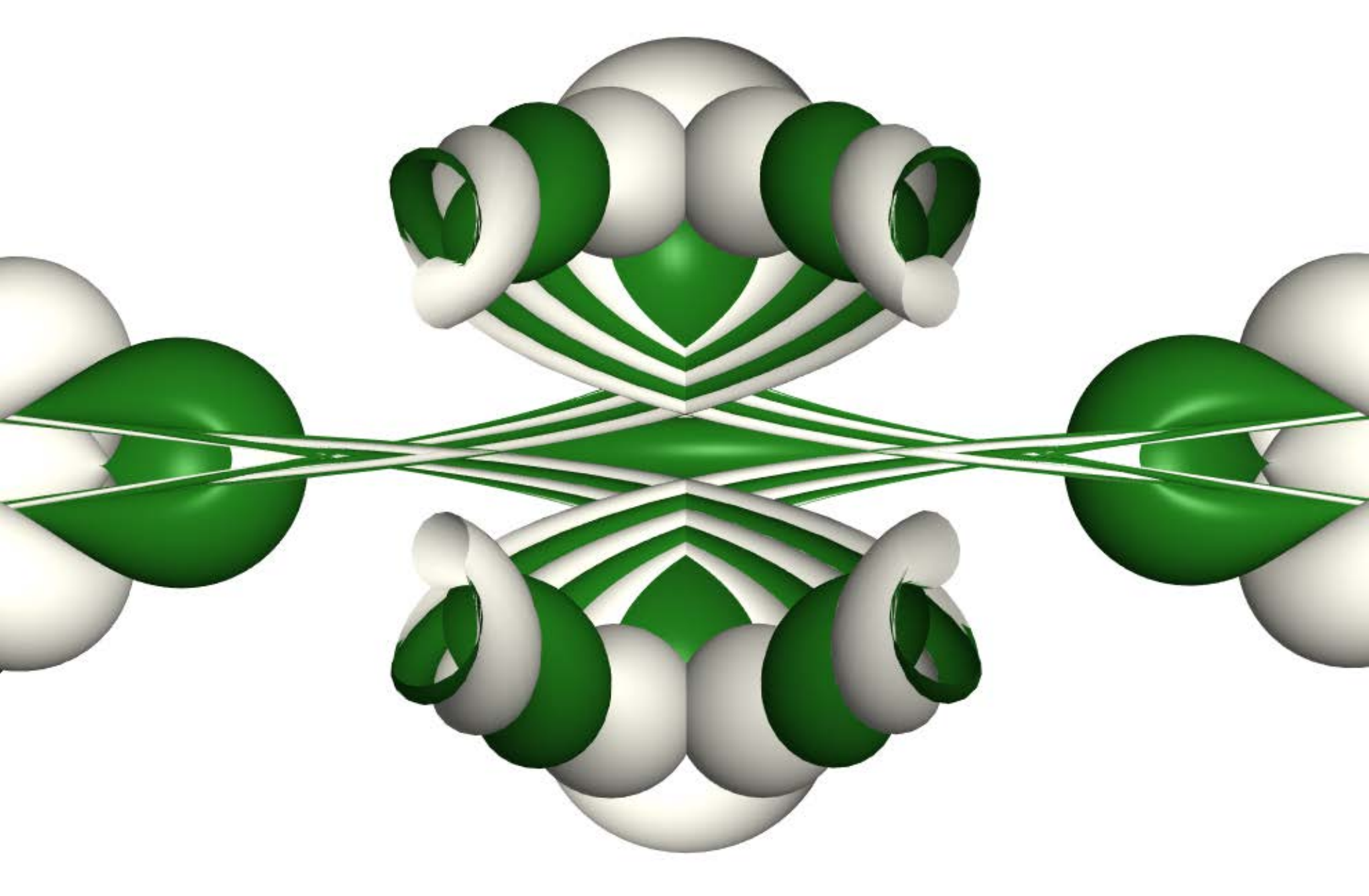}  \,\, &  \,\,
\includegraphics[height=33mm]{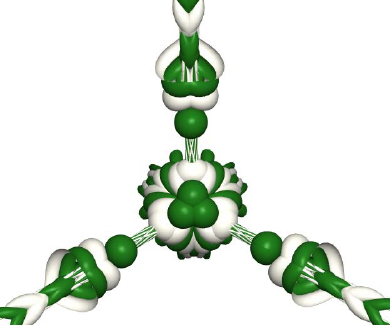}   \\
(\cos z^2 ,\sin  z^2  ) & (\cos z^2 , \frac{d^2}{\dd z^2}\sin z^2)  & (\cos z^3 , \frac{d^2}{\dd z^2}\sin z^3 ) 
\end{array}
$
\caption{ Spherical surfaces with potential data  $(a(z),b(z))$ (Theorem \ref{symmetrythm}).}
\label{figureB}
\end{figure}

In Section \ref{gcpsection} we give the solutions for the general geometric Cauchy problem 
(Theorem \ref{thm1}) and the \emph{singular} GCP (Theorems \ref{sgcpthm1} and \ref{thmgeneral}).
Theorem \ref{thm1} gives the holomorphic potentials for the solution spherical surface
containing a prescribed curve with surface normal prescribed along the curve. 

For the singular problems, in Theorem \ref{sgcpthm1}, the prescribed data is a regular real
analytic curve, and the solution is  the unique spherical frontal containing this curve as a cuspidal edge.
We also prove (Theorem \ref{beaksthm}) that, at a point where the curvature of the curve 
vanishes to first order, a \emph{cuspidal beaks} singularity occurs.

In Theorem \ref{thmgeneral}, the prescribed singular curve need not be regular, and this construction
includes cone points, swallowtails and cuspidal butterflies.  In Section \ref{conessection},
we rephrase this theorem in terms of the prescribed normal, relating the result to
results in \cite{ghm}.  This allows one to compute spherical fronts with embedded cone points
from an arbitrary real analytic closed convex curve in $\SSS^2$.

In Section \ref{generatingsection} we discuss the idea of using curves as geometric generators
for spherical surfaces and compute several examples.  Numerical implementations of the
DPW method \cite{DorPW}, such as that used here, have already been discussed in the literature, e.g.~\cite{KMS}, 
and so we do not discuss implementation issues.  

In the final section, we discuss open questions, such as the topology of the maximal immersed spherical surface
with a finite order rotational symmetry. We also consider the problem of the \emph{bifurcations}
of generic one parameter families of spherical fronts.  These are in some sense the second most
common singularities one expects to encounter, as they are the unstable singularities in a 
\emph{generic} 1-parameter family.   For the case of general fronts in $\real^3$,
these have been classified in \cite{arnoldetal}.   The list of bifurcations for \emph{spherical} fronts
is certainly not the same.   We propose a plausible list of bifurcations, based on our solutions
to the singular geometric Cauchy problem.

Lastly, we include, in Appendix \ref{appendix}, a streamlined account of the solution of the
geometric Cauchy problem for CMC surfaces, providing an explicit formula for the potential for the solution, which is
lacking in the original paper \cite{bjorling}.

\section{Preliminaries}    \label{background}
In this section we briefly describe the relation between spherical surfaces and CMC surfaces, and the loop group methods
we will use.  Details here are kept to the minimum needed. For more on spherical surfaces and harmonic maps, the reader
could consult \cite{ghm}. For the loop group theory used here, see ~\cite{bjorling}.

\subsection{Spherical surfaces and CMC surfaces} \label{parallelsection}
Suppose that $g$ is a regular parameterized surface with unit normal $N$ and
 Gauss and mean curvature $K_1$ and $H_1$ respectively, with $H_1$ a non-zero constant.
Define a parallel surface by
\[
f(x,y) = g(x,y) + \frac{1}{2H_1}N(x,y).
\]
Then 
$f_x \times f_y = (K_1/(4H_1^2)) g_x \times g_y$.
Thus $f$ is immersed precisely on the set $K_1 \neq 0$.
On this set, $f$ and $g$ have the same unit normal, and $f$ has Gauss and mean curvature functions respectively given by
\[
K = 4 H_{1}^2, \quad \quad H= \frac{2H_1(2H_1^2-1)}{K_1}.
\]

	\begin{figure}[ht]
\centering
$
\begin{array}{cc}
\includegraphics[height=38mm]{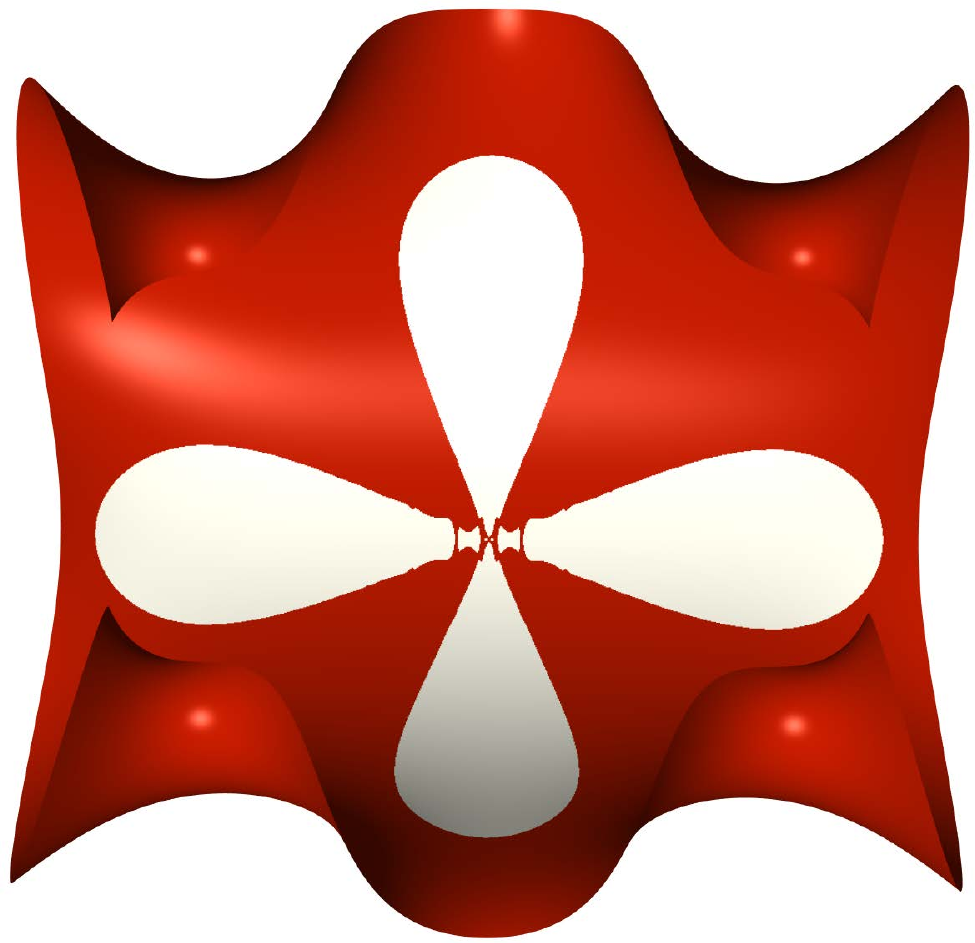} \quad  \quad & \quad
\includegraphics[height=38mm]{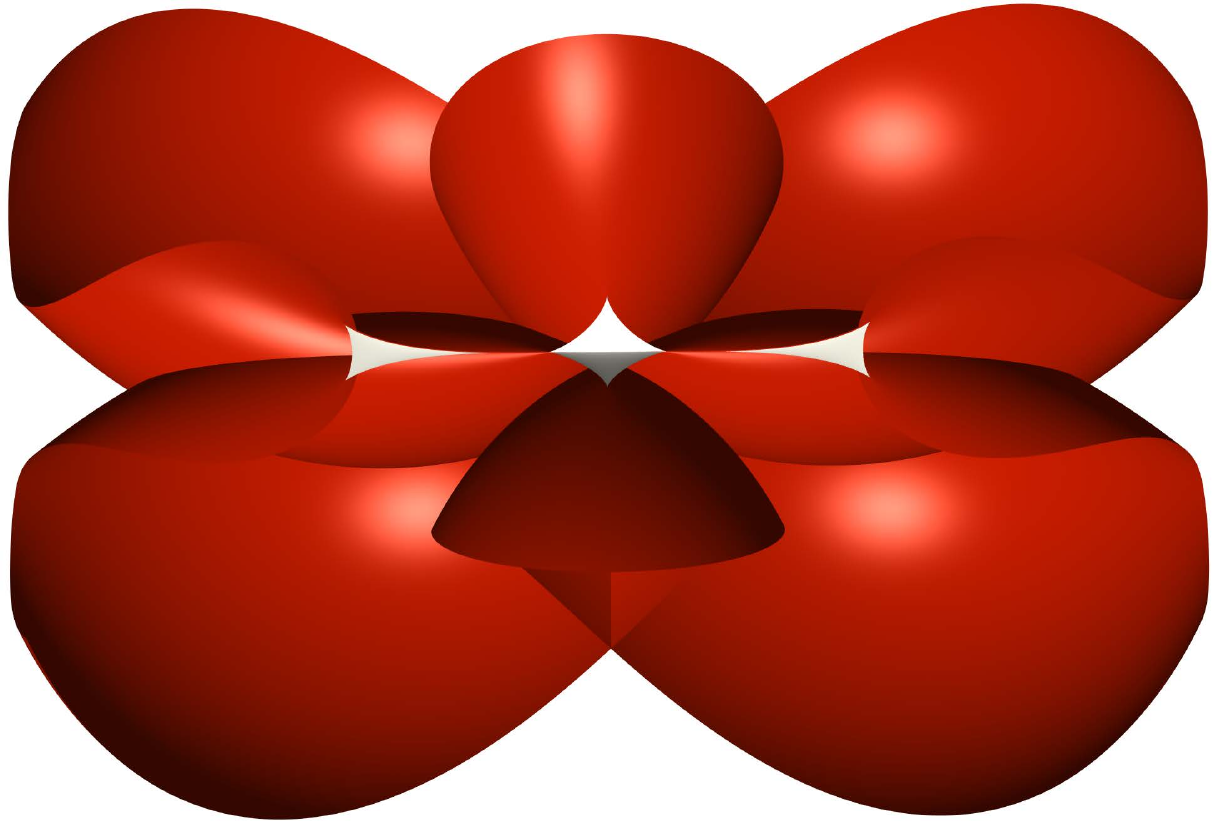}  
\end{array}
$
\caption{A  CMC $1/2$ surface (left) and the parallel spherical surface, plotted with the same color map.
Larger regions  are plotted in Figure \ref{figure8}.}
\label{figure0}
\end{figure}

Figure \ref{figure0} shows (left) a part of the unique CMC $1/2$ surface containing the plane curve with curvature function
$\kappa(s) = 1-s^4$ as a geodesic (see Theorem \ref{bjorlinthm} below). The CMC surface is colored by the sign of the Gauss
curvature.  To the right is the image of the same coordinate patch, with precisely the same color map, on the 
parallel spherical surface.  As stated above, this surface 
must have singularities
where the Gauss curvature of the parallel CMC surface changes sign, and one can indeed see that
there are cusp lines where the color of the surface changes.
 This example also shows clearly that it is not very intuitive to guess
what the parallel spherical surface to a given CMC surface will look like.

Conversely, if $f$ is a regular surface with constant positive  Gauss curvature $K$ and unit normal $N$, the parallel surface
$g_\mp=f \mp N/\sqrt{K}$ is immersed precisely on the set $H \neq \mp \sqrt{K}$,
 has the same unit normal as $f$,  and mean curvature $H_1 = \pm \sqrt{K}/2$.
In this way, spherical surfaces of constant Gauss curvature $K>0$ are in one to two correspondence with 
surfaces of constant mean curvature $H_1= \pm \sqrt{K}/2$.

\subsection{Harmonic maps into $\SSS^2$}
Let $\Omega$ be a simply connected open subset of $\C$, with holomorphic coordinates $z=x+iy$.
A smooth map $N: \Omega \to \SSS^2$ is harmonic if and only $N \times (N_{xx} + N_{yy} ) = 0$, i.e.
\[
N \times N_{z \bar z} = 0.
\]
This condition is also the integrability condition for the equation
\beq \label{sphericalsurface}
f_z = i \, N \times N_z,
\eeq
and more generally $f_z = i \, N \times N_z + a N_z$, where $a$ is any real constant.
 That is, $(f_z)_{\bar z}= (f_{\bar z})_z$ if and only if $N \times N_{z \bar z} = 0$.
Hence, given a harmonic map $N$, we can integrate the equation \eqref{sphericalsurface} to obtain a 
smooth map $f: \Omega \to \real^3$, unique up to a translation.  

A differentiable map $h: M\to \real^3$ from a surface into Euclidean space is called a \emph{frontal}
if there is a differentiable map $\mathcal{N}:M \to \SSS^2 \subset \real^3$ (or, more generally into ${\mathbb RP}^2$) such that $\dd h$ is orthogonal to $\mathcal{N}$.  The map is called a \emph{(wave) front} if
the Legendrian lift $(h, \mathcal{N}): M \to \real^3 \times \SSS^2$ is an immersion. 
Clearly, $f$ as defined above is a frontal with Legendrian lift $L:=(f,N)$.  
It follows from \eqref{sphericalsurface} that $f$ is regular if and only if $N$ is regular. At regular points
the first and second fundamental forms for $f$ are
\beqas
\mathcal{F}_I &=& |N \times N_y|^2 \, \dd x^2 + 2\langle N \times N_y, - N \times N_x \rangle  \, \dd x \dd y
   + |N \times N_x|^2 \, \dd y^2, \\
	\mathcal{F}_{II} &=& \langle N , N_x \times N_y \rangle \, (\dd x^2 + \dd y^2).
	\eeqas
Thus the metric induced by the second fundamental form is conformal with respect to the conformal structure on
$\Omega$.   Moreover, the expression for the Gauss curvature of $f$
simplifies to $1$. That is, $f$ has constant curvature $K=1$ wherever it is regular.   
Conversely, since the Gauss map of a spherical surface is harmonic with respect to the metric
induced by the second fundamental form, all spherical surfaces are obtained this way.
 Let us call the map $f$ the \emph{spherical
frontal} associated to the harmonic map $N$.

\subsection{The loop group representation}
Identify $\SSS^2=SU(2)/K$, where $K$ is the diagonal subgroup, with projection $\pi: SU(2) \to \SSS^2$ given by $\pi(X)=\Ad_X e_3$,
where 
\[
e_1 = \frac{1}{2}\bbar 0 & -i \\ -i & 0 \ebar, \quad e_2= \frac{1}{2} \bbar 0 & 1 \\ -1 & 0 \ebar \quad  e_3= \frac{1}{2} \bbar i & 0 \\ 0 & -i \ebar,
\]
 are an orthonormal basis for
$\real^3\equiv \mathfrak{su}(2)$ with inner product $\langle X, Y \rangle = -2 \trace(XY)$.  This choice of inner-product is
convenient because then the cross-product
in $\real^3$ is given by $a \times b = [a,b]$.

Let $N: \Omega \to \SSS^2$ be a smooth map, and let $F: \Omega \to SU(2)$ be any lift of $N$, i.e.~$N = \Ad_F e_3$.
 Let $\mathfrak{su}(2)=\mathfrak{k} + \mathfrak{p}$
be the decomposition corresponding to $SU(2)/K$, i.e.~$\mathfrak{k}=\hbox{span}(e_3)$ and $\mathfrak{p} = \hbox{span}(e_1,e_2)$.
The Maurer-Cartan form of $F$ then decomposes as
\[
\alpha := F^{-1} \dd F = (U_\mathfrak{k} + U_\mathfrak{p}) \dd z + ( -\bar U^t_\mathfrak{k} -\bar U^t_\mathfrak{p}) \dd \bar z.
\]
Then $N_z = \Ad_F[U_\mathfrak{p}, e_3]$, and equation \eqref{sphericalsurface} is equivalent to
\beq \label{frame1}
f_z = i \Ad_F U_\mathfrak{p}, \quad f_{\bar z} = i \Ad_F \bar U^t_{\mathfrak{p}}.
\eeq
The harmonic map equation $N \times N_{z \bar z}=0$ is equivalent to the equations
\beq \label{harmeqns2}
\partial _ {\bar z}  \, U_\mathfrak{p} +[- \bar U^t_\mathfrak{k}, U_\mathfrak{p}]=0,
\quad  \quad
\partial _ {z}  \, \bar U^t_\mathfrak{p} +[ U_\mathfrak{k},  \bar U^t_\mathfrak{p}]=0,
\eeq
and these equations are satisfied if and only if the $1$-parameter family of $1$-forms 
\[
\hat \alpha:= U_\mathfrak{p} \lambda \dd z + U_\mathfrak{k}\dd z -\bar U^t_\mathfrak{k} \dd \bar z -
     \bar U^t_\mathfrak{p} \lambda ^{-1}\dd  \bar z 
		\]
		is integrable for all complex $\lambda$, i.e.~if and only if $\dd \hat \alpha + \hat \alpha \wedge \hat \alpha = 0$.
		Consequently, we can integrate the equation $\hat F^{-1} \dd \hat F = \hat \alpha$, 
		to obtain an \emph{extended frame}  $\hat F: \Omega \to \Lambda SU(2)_\sigma$ into the group of twisted loops in $SU(2)$.
		The initial condition for the integration, at a point $z_0 \in \Omega$, is given by 
		$\hat F(z_0)= \hbox{diag}(\sqrt{\lambda}, \sqrt{\lambda}^{-1}) F(z_0)  \hbox{diag}(\sqrt{\lambda}^{-1}, \sqrt{\lambda})$,
	 the twisted form of $F(z_0)$.  Then $F= \hat F \big|_{\lambda=1}$,  and $N= \Ad_{\hat F} e_3 \big|_{\lambda=1}$,
	so $\hat F$ is also a lift of $N$.   Finally, we can recover the associated spherical frontal $f$ by the Sym formula:
	\[
	f= \mathcal{S}(\hat F) = i \left( \lambda \frac{\partial \hat F }{\partial \lambda}  \hat F^{-1}  \right)_{\lambda=1},
	\]
	because $f$ given by this formula satisfies the equation $f_z = i \Ad_F U_\mathfrak{p}$.  This construction does not
	depend on the choice of frame $F$ for $N$, and changing the integration point $z_0$ for $\hat F$ results 
	only in a translation to the solution $f$.
	
	\subsection{The DPW method}  \label{dpwsection}
	The above extended frame is identical to the extended frame for CMC surfaces, which has been discussed in many places.
	The construction for CMC surfaces differs only in the addition of a normal term in the Sym formula, because these
	are parallel surfaces to each other. The DPW method \cite{DorPW} allows one to construct all solutions from holomorphic data,
	so-called \emph{holomorphic potentials} of the form 
	\[
	\hat \eta = \sum_{n=-1}^\infty A_n \lambda^n \dd z,
	\]
	where $A_n$ are holomorphic $\mathfrak{sl}(2,\C)$-valued functions satisfying the twisting condition:
	$A_{2n}$ diagonal, $A_{2n+1}$ off-diagonal.  An extended frame is obtained by
	integrating $\hat \Phi ^{-1} \dd \hat \Phi = \hat \eta$, with $\hat \Phi(z_0)=I$ and then the (pointwise on $\Omega$)
	Iwasawa decomposition
	\beq \label{iwasawa}
	\hat \Phi  = \hat F \hat B_+, \quad \quad \hat F(z) \in \Lambda SU(2)_\sigma, \quad \hat B_+(z) \in \Lambda ^+ SL(2,\C),
	\eeq
	where $\Lambda^+G$ denotes the subgroup of loops that extend holomorphically to the unit disc in $\lambda$. The Iwasawa factorization $X= U B_+$ of  a loop $X$ is unique if we choose the constant term
	of $B_+$ to be real, i.e.~$B_+ = \hbox{diag}(\rho, \rho^{-1}) + O(\lambda)$, where $\rho \in \real$.
	
	An important difference between DPW for CMC surfaces and for spherical surfaces is \emph{regularity}.
	For a CMC surface, it is enough to choose a holomorphic potential that has the property that the $(1,2)$
	component function of the matrix function $A_{-1}$ does not vanish.  For spherical surfaces, the condition that
	$N$ (and hence $f$) is immersed works out (using $N_z = \Ad_F[U_\mathfrak{p}, e_3]$ or 
	\eqref{frame1})  to be that the upper right and lower left components
	of the matrix $U_\mathfrak{p}$ differ in absolute value. Note also that $\dd f=0$ at a point
	if and only if $U_\mathfrak{p}=0$.
Regularity can be guaranteed in a neighbourhood of the integration point by choosing
	$A_{-1}$ to have the same property, but in general the property will fail at some other points.

	\subsection{Bj\"orling's problem}   \label{bjorlingsection}
	A holomorphic potential for $\hat F$ is not unique, and this gives one the opportunity to use different potentials
	for specific purposes.  A method for solving the generalization of Bj\"orling's problem to non-minimal CMC 
	surfaces in given in \cite{bjorling},
	by defining the \emph{boundary potential} along a curve: suppose that $\hat F(x,0)=\hat F_0(x)$ is known along the curve $y=0$.
	Then define $\hat \eta$ to be the holomorphic extension of 
		\beq \label{F0mcform}
		\hat F_0^{-1} \dd \hat F_0 = \left(U_\mathfrak{p} \lambda + U_\mathfrak{k} -\bar U^t_\mathfrak{k}  -
     \bar U^t_\mathfrak{p} \lambda ^{-1} \right) \dd  x,
		\eeq
	away from the curve.  The above formula is obtained from $\hat \alpha$ by observing that $\dd z = \dd \bar z = \dd x$ along 
	the real line.
For this potential, the Iwasawa factorization is trivial along the curve, and so the solution obtained from it satisfies $\hat F(x,0)=\hat F_0(x)$,
	and is the unique solution to  the given Cauchy problem.

	\section{Spherical surfaces with branch points and rotational symmetries}   \label{symmetriessection}
	The simplest kind of holomorphic (or, in general, meromorphic) potential is a \emph{normalized potential}, which only has one term
	in the Fourier expansion:
	\beq \label{normpotential}
	\hat \eta = \bbar 0 &  a(z) \\ b(z) & 0 \ebar \lambda^{-1} \dd z.
	\eeq
The normalized potential is determined uniquely by a choice of \emph{basepoint}
$z_0$, via an extended frame $\hat F$ satisfying $\hat F(z_0)=I$, and a 
meromorphic frame $\hat F_-$ obtained from the Birkhoff decomposition
$\hat F (z)= \hat F_- (z)\hat F_+(z)$, with $\hat F_-(\lambda=\infty)=I$.
Then $\hat \eta = \hat F_-^{-1} \dd \hat F_-$ is a normalized potential. Note that $a$ and $b$
have no poles in a neighbourhood of the basepoint. 

\subsection{Local regularity and examples with branch points}
We say that $z_0$ is a \emph{singular point} for $f$ if $\hbox{rank} (\dd f(z_0)) <2$.
 Here is a local characterization of the regularity of  a spherical surface 
 in terms of holomorphic normalized potential data $(a(z),b(z))$:
\begin{lemma} \label{regularitylemma}
Let $f$ be a spherical surface and $\hat \eta = \hbox{off-diag}(a(z), b(z))\lambda^{-1} \dd z$ be the normalized potential  with basepoint $z_0$. Then:
\beqas
\hbox{rank} ( \dd f(z_0)) =1  \quad  & \Leftrightarrow  & \quad   |a(z_0)|=|b(z_0)| \neq 0, \\
\dd f(z_0)=0  \quad  & \Leftrightarrow &  \quad a(z_0) = b(z_0) = 0.
\eeqas
  Suppose now that $|a(z)| \not \equiv |b(z)|$. If $z_0$ is a singular point then either: 
\begin{enumerate}
 \item  \label{regitem1}
The derivative $\dd f(z_0)$ has rank $1$ and the singularity is not isolated, or
\item \label{regitem2}
 The derivative $\dd f(z_0)$ has rank $0$ and the singularity is isolated.
\end{enumerate}
In the second case, the singularity is a \emph{branch point}, defined to be a point where 
the harmonic map $N$ can be expressed in some local coordinates as $z \mapsto z^{k}$ for 
some integer $k\geq 2$.
\end{lemma}
\begin{proof}
As remarked above, the conditions on  non-regularity and the vanishing of the derivative of $f$
are encoded in the off-diagonal components
of $U_\mathfrak{p}$.  These have the same value as $a$ and $b$ at the normalization point. 
Items \ref{regitem1} and \ref{regitem2} follow from the properties of 
holomorphic maps. 
Since the harmonic map equations are $f_x = N \times N_y$ and $f_y=-N \times N_x$, 
the rank of $f$ is the same as the rank of $N$.
John C.~Wood proved \cite{wood1977} that the isolated singularities of such a harmonic
map are branch points.
\end{proof}

Three branched examples are computed and displayed in Figure \ref{figurebranch}.
Note that some branch points are ``removable'' in that the germ of the map around $z_0$
is a branched covering of a smooth surface. Take, for example, $a(z)=z$ and $b(z)=10 z$ (or the other
way around). Setting $w=z^2/2$ we have $(z, 10 z) \dd z = (1, 10) \dd w$, and the potential
$\hat \eta = \hbox{off-diag}(1, 10) \lambda^{-1} \dd w$ generates a smooth surface. 
Galvez et al.~ \cite{ghm} prove that a branch point is removable if and only if the mean curvature
is bounded around the point.  We will consider rank \emph{one} singularities below, using a different 
type of potential that is more suitable for data along a curve.

\begin{figure}[ht]
\centering
$
\begin{array}{ccc}
\includegraphics[height=30mm]{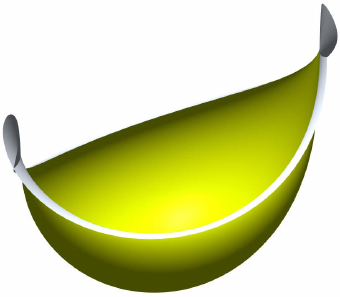}  \quad & \quad 
\includegraphics[height=30mm]{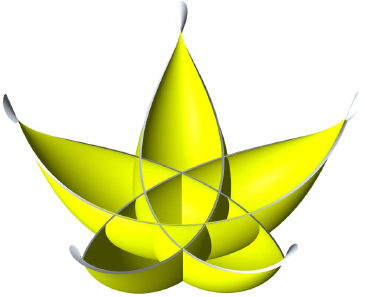}  \quad &  \quad
\includegraphics[height=30mm]{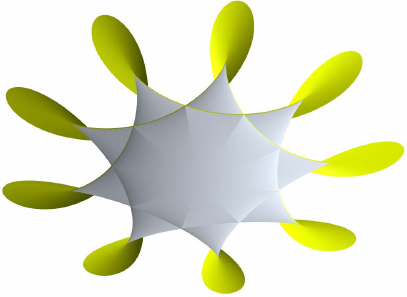}  \vspace{1ex}  \\
(z, \, 10 \,z) & (z, \,10\, z^2) &    (z, \,  10 \,z^6)  
\end{array} 
$ 
\caption{Branched spherical surfaces with normalized potentials $(a(z),b(z))$.}
\label{figurebranch}
\end{figure}

\subsection{Finite order rotational symmetries}   
As the examples in Figure \ref{figurebranch} suggest, 
we can use normalized potentials to construct spherical surfaces with rotational symmetries: 
 let $n$ be an integer and set
$\theta_n:= 2 \pi /n$.  An  immersion $g: \Omega \to \real^3$ has
a \emph{fixed-point rotational symmetry of order $n$} if there exist conformal
coordinates $z$ on $\Omega$ such that $e^{i \theta_n }\Omega = \Omega$ and
such that 
\[
g(e^{ i \theta_n} z) = R_{\theta_n} g(z)
\]
for all $z \in \Omega$, 
where $R_{\theta_n}$ is the rotation of angle $\theta_n$ about some axis.
In \cite{mincmc}, it is
shown (Lemma 7.3) that a CMC $1/2$ surface has such a  rotational symmetry 
if and only if the meromorphic functions $a$ and $b$ in the 
normalized potential, with basepoint $z_0=0$, have Laurent expansions of the form:
$a(z) = \sum_j a_{nj} z^{nj}$ and $b(z) = \sum_j b_{nj-2} z^{nj-2}$.
The same statement is valid for a spherical surface, because if $g$ has the
above symmetry, then so does the unit normal $N$, and consequently the
parallel surface $f= g+N$.  In a neighbourhood of the basepoint $z_0$, at which
$\hat F(z_0)=\hat F_-(z_0)=I$, the functions $a$ and $b$ are holomorphic.
Adapting Lemma 7.3 of \cite{mincmc} slightly to the fact that there are two parallel
CMC surfaces to a spherical surface, one of which may have a branch point, we
conclude:
\begin{theorem} (Corollary of Lemma 7.3 in \cite{mincmc}).  \label{symmetrythm}
Let $n\geq 2$ be an integer, and $\theta_n:= 2 \pi /n$.
Let $\Omega \subset \C$ be an open set with $0 \in \Omega$ and such that
 $e^{i \theta_n} \Omega = \Omega$.
Suppose that a spherical surface $f: \Omega \to \real^3$ has a fixed-point rotational symmetry 
of order  $n$.  Let $\hat \eta$ be the normalized potential \eqref{normpotential}
 with respect to basepoint $z_0=0$. Then the pair $(a(z), b(z))$ has Taylor expansion
about $0$ of one of the following two forms:
\[
\left( \sum_{j=0}^\infty a_{nj} \, z^{nj} , \,\,   \sum_{j=1}^\infty  b_{nj-2} \, z^{nj-2} \right), 
\quad \hbox{or} \quad
 \left(    \sum_{j=1}^\infty  a_{nj-2} \, z^{nj-2} ,  \, \,\sum_{j=0}^\infty b_{nj} \, z^{nj}   \right).
\]
 Conversely, any such pair of holomorphic functions $a$ and $b$ generate
a spherical surface with a  fixed-point rotational symmetry 
of order  $n$.
\end{theorem}

\begin{figure}[ht]
\centering
$
\begin{array}{cccc}
\includegraphics[height=26mm]{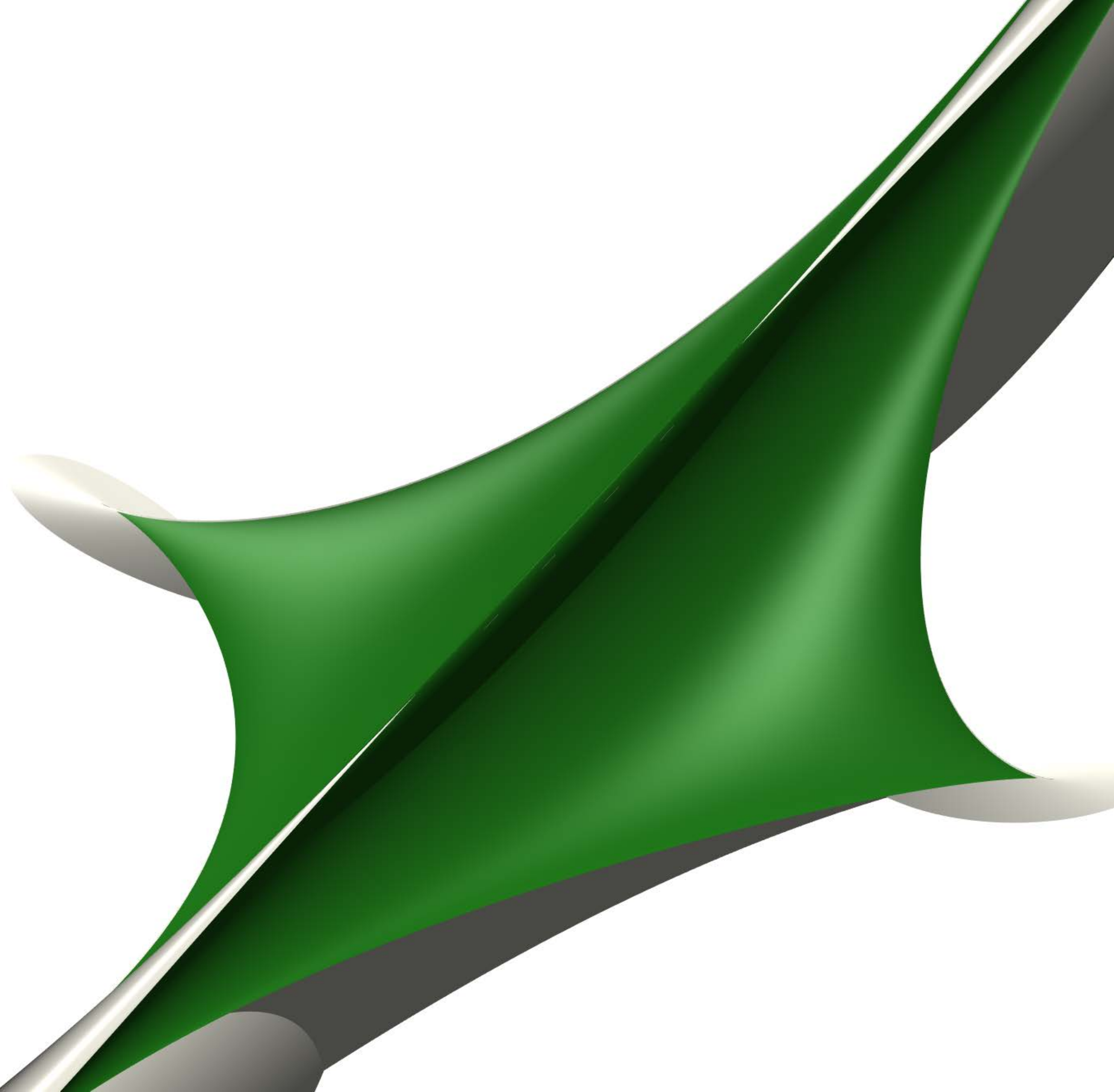}  &  \quad
\includegraphics[height=26mm]{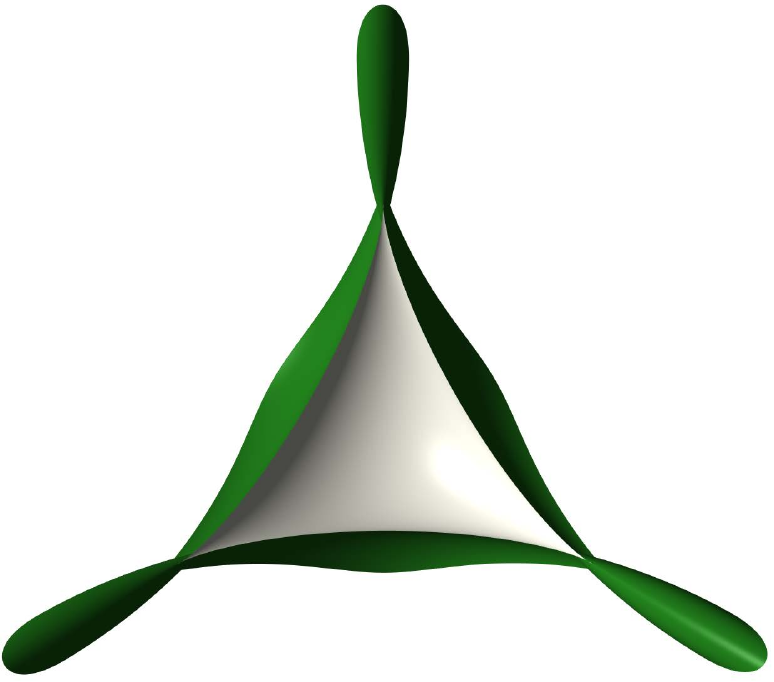} \quad &  \quad
\includegraphics[height=26mm]{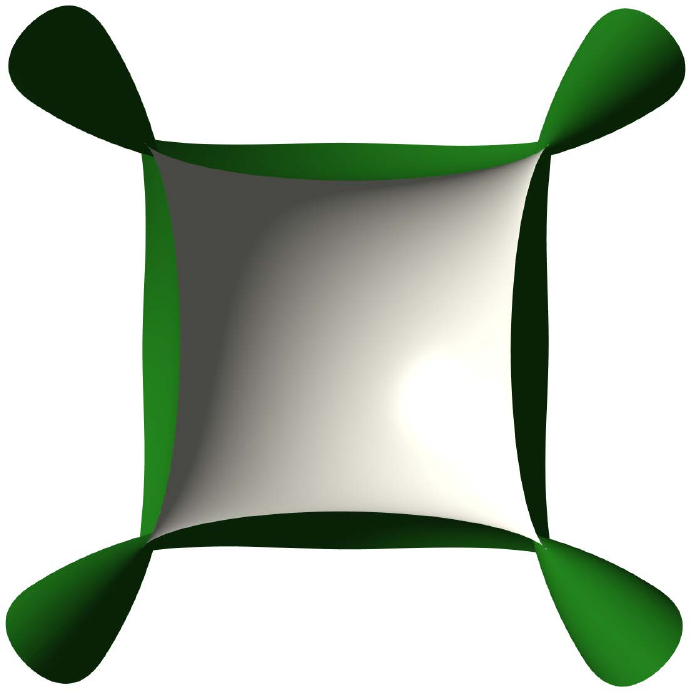}   & \quad 
\includegraphics[height=26mm]{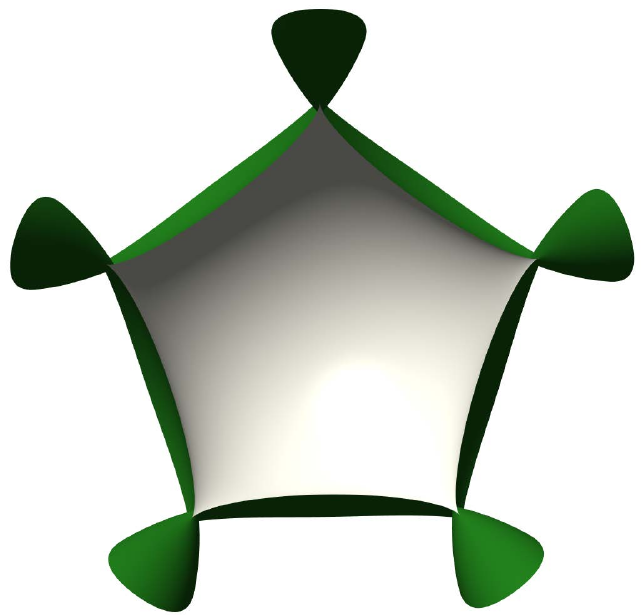}    \\
(1 + z^2, 1) & (1 + z^3, z) &    (1 + z^4, z^2)  &    (1 + z^5, z^3)  
\end{array} 
$ \vspace{2ex} \\
$
\begin{array}{ccc}
\includegraphics[height=33mm]{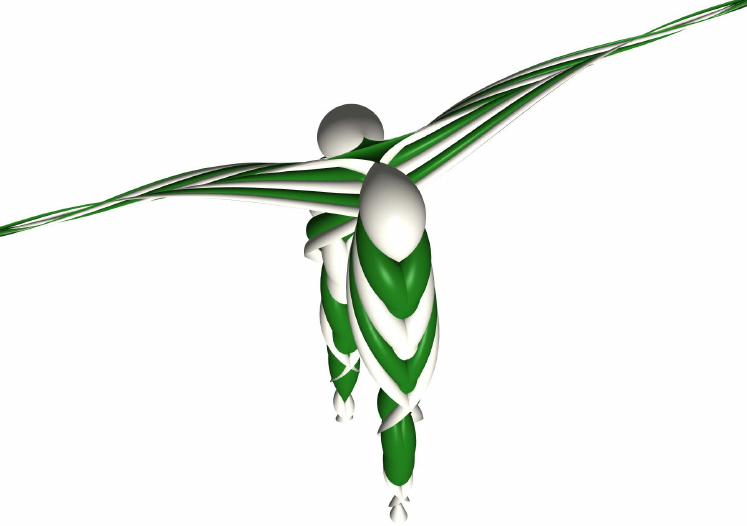}  \quad &  
\includegraphics[height=33mm]{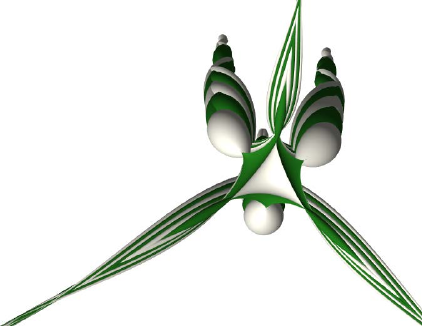}  &  \quad
\includegraphics[height=33mm]{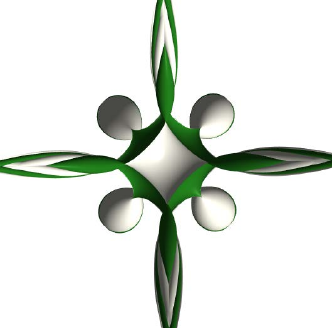}    \\
(1 + z^2, 1) & (1 + z^3, z) &    (1 + z^4, z^2)  
\end{array} 
$ \vspace{2ex} \\
$
\begin{array}{ccc}
\includegraphics[height=33mm]{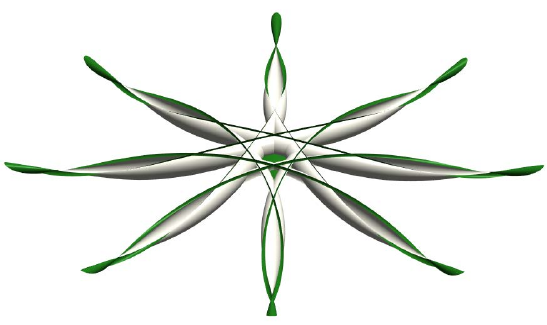}   &  \quad
\includegraphics[height=33mm]{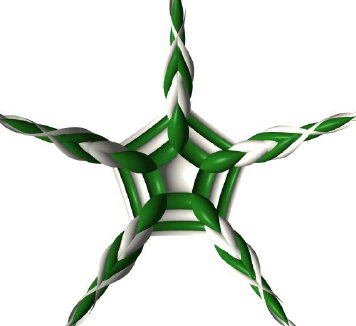} \quad & 
\includegraphics[height=33mm]{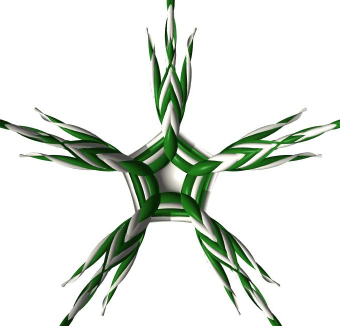}   \\
  (z^4, z^2)  & (1 , z^3) &  (1+ z^5, z^3+z^8)
\end{array}
$
\caption{Spherical surfaces computed from normalized potentials $(a(z),b(z))$.}
\label{figureA}
\end{figure}

  Examples with various polynomial choices for 
  $(a(z),b(z))$ are shown in Figure \ref{figureA}.  The cases $(1+z^2,1)$ and $(z^4,z^2)$  each
	have a singularity at $z=0$, because, for these functions $|a(0)|=|b(0)|$.

For polynomials $a(z)$ and $b(z)$, the \emph{degree} of the polynomials appears to determine
the number of ``legs'' or ``ends" on the surface.  The examples in Figure \ref{figureA}, as well as all
other polynomial examples we have checked,  follow the rule:
\[
\hbox{number of legs} = \deg(a) + \deg(b)+2.
\]

As non-polynomial examples we could consider, for an example with $n=4$,
$(a(z),b(z))= (\cos(z^2), \sin(z^2))$, (Figure \ref{figureB}, left). 
For general $n$, we could use
\[
a(z) = \cos(z^n), \quad \quad b(z) = \frac{\dd ^2}{\dd z^2}\sin(z^n).
\]
The cases $n=2$ and $n=3$ are shown in Figure \ref{figureB}, middle and right.

	\section{The geometric Cauchy problem for spherical surfaces}   \label{gcpsection}
	The idea of a geometric Cauchy problem is to specify sufficient geometric data along a space curve $f_0$ to generate
	a unique solution surface for some class of surfaces. This generalizes Bj\"orling's problem for minimal surfaces,
	where it is sufficient to prescribe the surface normal along the curve.  A spherical surface is analytic, and so we may
	always assume that $f_0$ is analytic.  Thus, if the curve is not a straight line, we can also assume that there is 
	a well-defined Frenet frame for the curve, and the curvature function $\kappa$ is real analytic and 
	can change sign at points where it vanishes \cite{willmore1959}.  We will consider three ways to prescribe a unique
	solution:  (a) stipulate that the curve is a geodesic in the solution surface, (b)  prescribe the surface normal along the curve,
	and (c) stipulate that the curve is a non-degenerate \emph{singular} curve.
	
	Suppose given a regular, arc-length parameterized,
	real analytic curve $f_0: J \to \real^3$, where $J$ is some open interval, and a vector field $N_0: J \to \SSS^2$,
	normal to $f_0$, i.e.~$\langle f_0^\prime(t), N_0(t) \rangle =0$.  
	We seek the (unique)
	spherical surface containing $f_0$ and with surface normal given, along $f_0$, by $N_0$.  For a given solution $f$,
	with a parameterization that is conformal with respect to the second fundamental form, 
	on an open set containing the initial curve we can always change our conformal coordinates so that the curve is
	given by $y=0$. We thus seek a solution $f(x,y)$ with $f(x,0)=f_0(x)$.  
	
	We want to find $U_\mathfrak{p}$ and $U_\mathfrak{k}-\bar U^t_\mathfrak{k}$ in 
	\eqref{F0mcform} in terms of the prescribed data.
	From \eqref{frame1}, we have, for a given solution,
	\[
	f_x = i \Ad_F(U_\mathfrak{p} + \bar U_\mathfrak{p}^t), \quad \quad
	f_y =  \Ad_F(-U_\mathfrak{p} + \bar U_\mathfrak{p}^t).
	\]
	Since $f_0$ is assumed regular, we can assume that $U_\mathfrak{p} + \bar U_\mathfrak{p}^t \neq 0$ along the curve.
	We can also assume, after a gauge, that the frame is chosen such that:
	\[
	 f_x = \Ad_F e_1,
	\]
	i.e.~that $  i(U_\mathfrak{p} + \bar U_\mathfrak{p}^t) = e_1$.  Expanding $U_\mathfrak{p}$ in the basis $e_1, \, e_2$ for $\mathfrak{p}$, this choice of frame gives 
	\[
	U_\mathfrak{p}= (a-\frac{1}{2}i)e_1 + b e_2, 
	\]
	where $a$ and $b$ are real-valued functions. Then
	\[
	f_x \times f_y = -2b \Ad_F e_3,
	\]
	so the surface is regular at precisely the points where $b \neq 0$.
	
		\subsection{The geodesic case}
		For clarity, we first treat the case that the curve is a geodesic in the solution surface.
	The condition that $f_0$ is a geodesic curve is that either the curvature $\kappa$ of 
	$f_0$ is zero or the surface normal $N$ coincides with the curve's normal.  Thus,	
 the geodesic condition is:
	\[
	f_{xx}= \kappa N = \kappa \Ad_F e_3, \quad \hbox{along } \{y=0\}.
	\]
With the choice of frame $F$ such that  $f_x = \Ad_F e_1$, with $U_\mathfrak{p}=(a-i/2)e_1 + be_2$, this works out 
to $f_{xx} = -2b \Ad_F e_3$, so 
\[
b=-\kappa/2, \quad \quad f_y= \Ad_F(-2ae_1+\kappa e_2),
\]
 and the surface is
immersed if and only if   $\kappa\neq 0$.  Assuming this,
denote the unit normal, binormal and torsion of the curve $f_0$ by 
	$\mathbf n$, $\mathbf b$ and $\tau$ respectively.
Then ${\mathbf n}= N=\Ad_F e_3$, so the binormal is
${\mathbf b}= -\Ad_F e_2$. Differentiating this, and writing $U_\mathfrak{k} = (c+di)e_3$, we have
\beqas
-\tau {\bf n} = {\bf b}_x &=& -\Ad_F[2c e_3 + 2ae_1-\kappa e_2, e_2] \\
&=&2c \Ad_F e_1 -2a \Ad_F e_3,
\eeqas
so $c=0$ and $a=\tau/2$.  Thus,
	\[
	U_\mathfrak{p} = \frac{\tau -i}{2}e_1 -\frac{\kappa}{2}e_2, \quad
	U_\mathfrak{k}-\bar U^t_\mathfrak{k} = 0, \quad
	-\bar U^t_\mathfrak{p} = \frac{\tau +i}{2}e_1 -\frac{\kappa}{2}e_2.
	\]
		Substituting into \eqref{F0mcform} for $\hat F_0^{-1} \dd \hat F_0$ we conclude:
	\begin{proposition}  \label{prop1}
	Let $f_0: J \to \real^3$ be a regular, arc-length parameterized, real analytic curve, with 
	non-vanishing curvature function $\kappa$, and torsion function $\tau$.
	Write $\kappa(z)$ and $\tau(z)$ for the holomorphic extensions of these
	functions to an open set $\Omega$ containing $J \times \{0\}$ in $\C$.    Then
	there is a unique spherical frontal $f: \Omega \to \real^3$, immersed on an open set containing
	$J \times \{0\}$, and such that $f(x,0) = f_0(x)$ is a geodesic curve in $f(\Omega)$.
	The solution $f$ is given by 
	the DPW method with holomorphic potential
	\[
	\hat \eta = \left(  \left(\frac{\tau(z) -i}{2}e_1 -\frac{\kappa(z)}{2}e_2 \right) \lambda  
		 + \left(\frac{\tau(z) +i}{2}e_1 -\frac{\kappa(z)}{2}e_2 \right) \lambda^{-1}  \right) \dd z.
	\]
	\end{proposition}

\begin{figure}[ht]
\centering
$
\begin{array}{cccc}
\includegraphics[height=26mm]{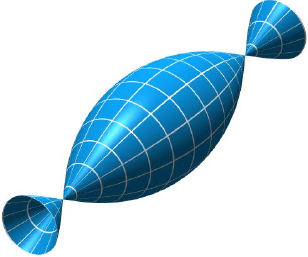}    \, \, \, & \, \, \,
\includegraphics[height=26mm]{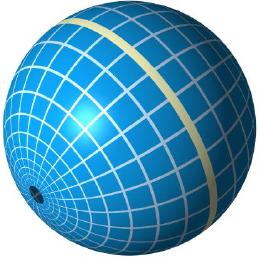}    \, \, \, & \, \, \,
\includegraphics[height=26mm]{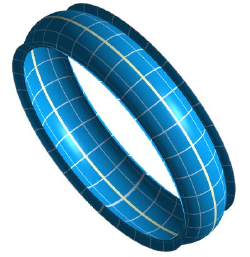}    \, \, \, & \, \, \,
\includegraphics[height=26mm]{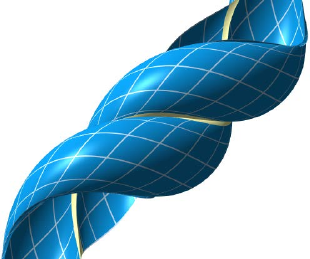}   
 \\
(\kappa, \tau)=(2,0) &  (\kappa,\tau)=(1,0) & (\kappa, \tau)=(1/2,0)  & (\kappa,\tau)=(1,1) 
\end{array}
$
\caption{Solutions of the geodesic GCP for various constant choices of $\kappa$ and $\tau$. }
\label{figure1}
\end{figure}

\begin{example} If the potential is constant then the geodesic will be a parallel with respect to the action of a 
$1$-parameter subgroup of the isometry group of $\real^3$, and the whole surface will be equivariant with
respect to this action. Some examples are computed and displayed in figure \ref{figure1}: the first three are
representative of the three types of spherical surfaces of revolution, and the last is equivariant with 
respect to screw-motions.
\end{example}

\begin{figure}[ht]
\centering
$
\begin{array}{ccc}
\includegraphics[height=25mm]{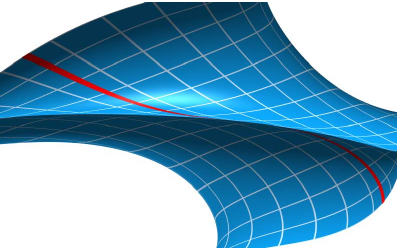}  \quad & \,\,
\includegraphics[height=25mm]{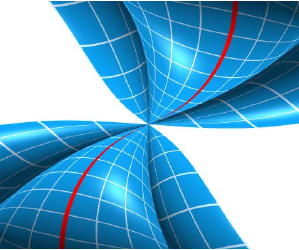}  \,\,  & \quad
\includegraphics[height=25mm]{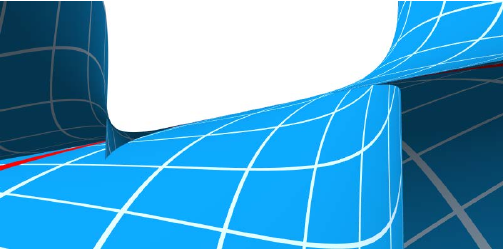}   
 \\
(\kappa(s), \tau(s))=(s,1) &  (\kappa(s), \tau(s))=(s,0) &  (\kappa(s), \tau(s))=(s^2,1) 
\end{array}
$
\caption{Examples of geodesics with inflection points. }
\label{figure2}
\end{figure}

\begin{example}  \label{sphericalexample2}
If we allow $\kappa$ to vanish, then there is still a spherical frontal generated by the potential, and there will
be a singularity at any point where $\kappa$ vanishes. Three examples are shown in Figure \ref{figure2}.
Singularities will be treated analytically below, but here we can observe that
the first case, $\kappa(s)=s$ and $\tau(s)=1$ represents the generic situation, and the geodesic
passes through a cuspidal edge at $(0,0)$.  The case $\kappa(s)=s$ and $\tau(s)=0$ is a special case, and this
is clearly a cone point.  The last case, where $\kappa$ vanishes to second order, is a degenerate singularity, 
with two cuspidal edges crossing.
\end{example}

	\subsection{The general geometric Cauchy problem}
	The solution to the geometric Cauchy problem for a given curve $f_0(x)$ with arbitrary 
	surface normal prescribed along the curve is obtained in the same way as Proposition \ref{prop1}.
	The only difference is that we replace the binormal $\bf b$ with the vector field $f_x \times N$,
	which is not parallel to $\bf b$ if $N$ is not parallel to $\bf n$.
	Thus the frame satisfies $f_x = \Ad_F e_1$, $- f_x \times N = \Ad_F e_2$ and $N=\Ad_F e_3$, and 
	\[
	f_{xx} = \kappa_g \Ad_F e_2 + \kappa_n \Ad_F e_3,
	\]
	where $\kappa_g$ and $\kappa_n$ are the geodesic and normal curvatures. 
	Then going through similar computations as before, we have $c=\kappa_g/2$ and $b=-\kappa_n/2$.
	Differentiating $N=\Ad_F e_3$ we find $a=\mu/2$, where
	$\mu=\langle f_x \times N, N_x\rangle$.  The surface regularity condition is $\kappa_n \neq 0$. In conclusion:
		\begin{theorem}  \label{thm1}
	Let $f_0: J \to \real^3$ be a regular, arc-length parameterized, real analytic curve
	and $N_0: J \to \SSS^2$ a real analytic map such that
	\[
	\langle f_0^\prime(x), N_0(x) \rangle =0, \quad \quad 
	\kappa_n(x):= \langle f_0^{\prime \prime}(x), N_0(x) \rangle \neq 0.
	\]
Set $\mu:= \langle f_0^\prime \times N_0, N_0^\prime \rangle$,
and $\kappa_g := \langle f_0^{\prime \prime}, N_0 \times f_0^\prime \rangle$.
Then the analogue of Proposition \ref{prop1} holds, and the unit normal $N$ of the solution
spherical frontal $f$ satisfies $N(x,0)=N_0(x)$.  
	The solution $f$ is given by 
	the DPW method with holomorphic potential
	\[
	\hat \eta = \left(  \left(\frac{\mu(z) -i}{2}e_1 -\frac{\kappa_n(z)}{2}e_2 \right) \lambda  
	 + \kappa_g(z) e_3
		 + \left(\frac{\mu(z) +i}{2}e_1 -\frac{\kappa_n(z)}{2}e_2 \right) \lambda^{-1}  \right) \dd z.
	\]
	\end{theorem}
	
	\begin{figure}[ht]
\centering
$
\begin{array}{ccc}
\includegraphics[height=30mm]{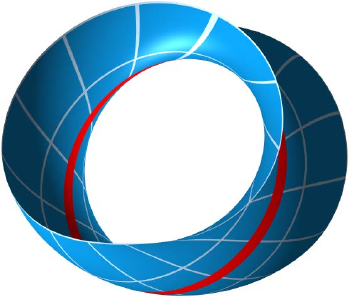}  \quad & \quad
\includegraphics[height=30mm]{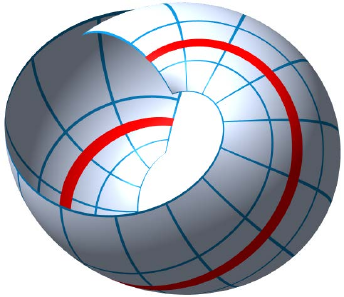}  \quad & \quad
\includegraphics[height=30mm]{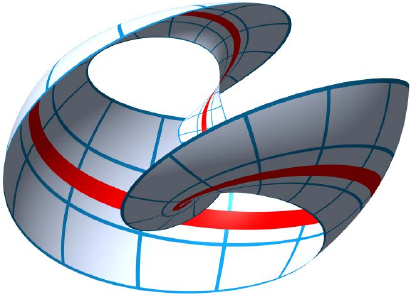}   \\
\hbox{$x$ interval } [ 0,2\pi ]   & \hbox{$x$ interval }   [0,2\pi]   & \hbox{$x$ interval }  [0,4\pi]  
\end{array}
$
\caption{Non-orientable spherical cylinder and parallel CMC surface.}
\label{figure3}
\end{figure}
\begin{example}  \label{nonorientableexample}
Taking 
\[
f_0(s)=(\cos(s),\sin(s),0), \quad \quad
N_0(s)=\cos(s/2)(0,0,1)+\sin(s/2)f_0(s),
\]
we have $N_0(s+2\pi)=-N_0(s)$, while
$f_0(s+2\pi)=f_0(s)$.  Since  the equation $f_z=i N \times N_z$ defining $f$ is invariant under the
$N \mapsto -N$,  the frontal $f$ is periodic in the $x$ direction with period $2\pi$. But 
the map $f:  (\real/2\pi \Z) \times \real \to \real^3$ is not orientable:
 the smooth harmonic map $N: \real^2 \to \SSS^3$ corresponding to the solution 
is well defined on  $(\real/4\pi \Z) \times \real$  but not on  $(\real/2\pi \Z) \times \real$.
 Now, we have
\[
\kappa_n(s)=-\sin(s/2), \quad \kappa_g(s)=\cos(s/2), \quad  \mu(s)=1/2.
\]
The surface obtained from this data has a singularity when $\kappa_n(s)=0$, 
that is, at $s=2k\pi$ for integers $k$, and is shown in Figure \ref{figure3}, left.
Although this cylinder is non-orientable (as a frontal),
topologically it is a cylinder, not a M\"obius strip.
The parallel CMC 1/2 surface is shown at center.  The third image effectively shows
both the parallel CMC $1/2$ and CMC $-1/2$ surfaces with respect to the unit normal
on the strip $0 < x< 2\pi$, which combine into a single oriented CMC cylinder.
\end{example}

\subsection{The singular geometric Cauchy problem}  
The condition that the curve $f_0(x) = f(x,0)$ is a singular
	set is $f_x \times f_y=0$,
	i.e.~$b=0$, and the singular set  is \emph{non-degenerate}, that is locally a regular curve in the coordinate domain, if and only if
	$\dd b \neq 0$.
The frame chosen above satisfies
	$U_\mathfrak{p}=  (a-\frac{1}{2}i)   e_1$ along such a singular curve.
	Now write $U_\mathfrak{k}=(c+ di)e_3$,  so
	\beqas
	f_{xx}  &=& \Ad_F [U-\bar U^t, e_1]  \\
	  &=& \Ad_F[2 c e_3 + 2a e_1, e_1] =  2c \Ad_F e_2.
	\eeqas
	Hence,  
	\[
	c = \frac{\kappa}{2}, \quad {\mathbf n} = \Ad_F e_2, \quad {\mathbf b} = \Ad_F e_3 = N_0.
	\]
	Similarly
	$- \tau \, {\mathbf n} = {\mathbf b}_x = \Ad_F[2 ce_3 + 2ae_1, e_3] = -2a \Ad_F e_2$, so 
	$a = \tau /2$. 	Thus, along $y=0$, we have
	\[
	U_\mathfrak{p} = \frac{\tau -i}{2}e_1, \quad
	U_\mathfrak{k}-\bar U^t_\mathfrak{k} = \kappa e_3, \quad
	-\bar U^t_\mathfrak{p} = \frac{\tau +i}{2}e_1.
	\]
	To find the non-degeneracy condition, $\partial_y b \neq 0$, we can substitute
	$U_\mathfrak{p} = (a-i/2)e_1 + b e_2$, with $a(x,0)=\tau(x)/2$ and $b(x,0)=0$, and
	$U_\mathfrak{k} = (c+i d) e_3$, with $c(x,0)=\kappa(x)/2$, into the harmonic map 
	equations \eqref{harmeqns2}.  The $e_2$ component of these equations gives
	$\partial_x b = d-2ac$ and $\partial_y b = c(1+4a^2)$.   Substituting the values of $a$, $b$ and
	$c$ along $y=0$, this becomes
	\beq  \label{dbexpressions}
	d(x,0)=\frac{\kappa(x) \tau(x)}{2}, \quad 
	\quad \frac{\partial b}{\partial y}(x,0) = \frac{\kappa(x)(1+\tau^2(x))}{2},
	\eeq
	and so the non-degeneracy condition is $\kappa \neq 0$.
	Substituting the above data into \eqref{F0mcform} for $\hat F_0^{-1} \dd \hat F_0$ we conclude:
	\begin{theorem}  \label{sgcpthm1}
	Let $f_0: J \to \real^3$ be a regular, arc-length parameterized, real analytic curve, with  curvature and torsion functions respectively $\kappa$ and $\tau$,  such that $\kappa \not \equiv 0$.   Write $\kappa(z)$ and $\tau(z)$ for the holomorphic extensions of these
	functions to an open set containing $J \times \{0\}$ in $\C$.    Then:
	\begin{enumerate}
	\item
	The unique solution to the singular geometric Cauchy problem for $f_0$ is given by
	the DPW method with holomorphic potential
	\[
	\hat \eta = \left( \frac{\tau(z) -i}{2} \lambda e_1  + \kappa(z) e_3  + \frac{\tau(z) +i}{2} \lambda^{-1} e_1 \right) \dd z.
	\]
	\item
	The singular set $C := J \times \{0\}$ is non-degenerate at a point $x_0$ if and only if $\kappa(x_0) \neq 0$. In the neighbourhood
	of such a point, $C$ is locally diffeomorphic to a cuspidal edge.  All cuspidal edges on spherical frontals
	are obtained this way.
	\end{enumerate}
	\end{theorem}
\begin{proof}
The uniqueness of the solution can be shown in the same way as in \cite{bjorling}.
The only details remaining are those concerning cuspidal edges.
The curve $f(x,0)$ is a cuspidal edge, because the null vector field, i.e.~the kernel of $\dd f$,
	is given by $\tau \partial _x + \partial _y$ and this is transverse to the singular curve
	 -- a criterion for a cuspidal edge \cite{krsuy}.
	We began the discussion with an arbitrary singular curve on a spherical frontal, and 
so all cuspidal edges  are obtained this way.
\end{proof}
	
			\begin{figure}[ht]
\centering
$
\begin{array}{cccc}
\includegraphics[height=26mm]{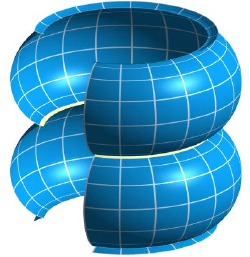}  \, & \,
\includegraphics[height=26mm]{images/sing2.pdf}  \,\, & \, \,
\includegraphics[height=26mm]{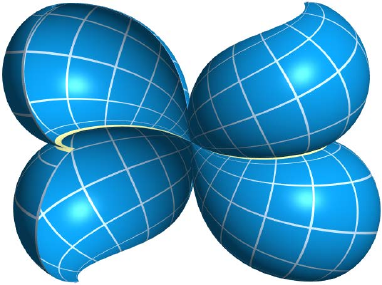}    \,& \,
\includegraphics[height=26mm]{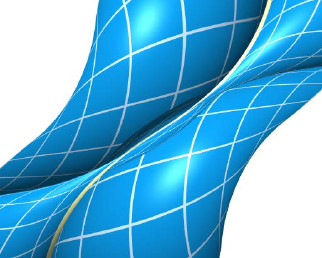}  \\
(\kappa, \tau)=(1,0) &  (\kappa,\tau)=(1,1) & (\kappa, \tau)=(s,0) & (\kappa,\tau)=(s,1)
\end{array}
$
\caption{Solutions of the singular GCP for various choices of $\kappa(s)$ and $\tau(s)$. }
\label{figure4}
\end{figure}

	Some examples are shown in Figure \ref{figure4}.
	The example $(\kappa, \tau)=(s,1)$ apparently has a cuspidal beaks singularity at $z=0$, as does
	the example generated by the normalized potential $(a,b)=(1+z^2,1)$ in Figure \ref{figureA}.
	The cuspidal beaks singularity is locally diffeomorphic to the map
	$\real^2 \to \real^3$ given by $(u,v) \mapsto (3u^4-2u^2v^2, u^3-uv^2,v)$ at $(0,0)$ 
	(see, e.g. \cite{fukuihasegawa}).  There is an identification criterion for cuspidal beaks of Izumiya, Saji and Takahashi \cite{izsata},  conveniently stated as Theorem 3.2 in \cite{fukuihasegawa}:
	a degenerate singularity of a front $f$ with unit normal $N$ is locally diffeomorphic
	to a cuspidal beaks around $p$ if and only if $\hbox{rank}(\dd f_p)=1$, $\det(\hbox{Hess} \, \mu(p))<0$
	and $\eta^2 \mu(p) \neq 0$, where $\mu := \langle f_x \times f_y , N \rangle$ and
	$\eta$ is the null vector field. Using this we can prove:
	\begin{theorem}  \label{beaksthm}
	Let $f_0$ and $\hat \eta$ be as in Theorem \ref{sgcpthm1}. Then the singular
	set is locally diffeomorphic to a cuspidal beaks at $(x_0, 0)$ if and only if
	$\kappa(x_0)=0$, $\kappa^\prime(x_0) \neq 0$ and $\tau(x_0) \neq 0$.
	\end{theorem}
	\begin{proof}
	We use the criteria given above:  the rank of $\dd f$ is certainly $1$, and a spherical
	frontal is a front if and only if its derivative has rank at least $1$.  The null direction 
along $y=0$ is	$\eta = \tau \partial _x + \partial _y$ and we have $\mu = -2b$.
Hence, along $y=0$, we have 
$ \eta \mu = -2(\tau (x) \partial_x b + \partial_y b)$ and 
\[
\eta(\eta \mu) = -2\left( \eta(\tau) \frac{\partial b}{\partial x } + 
 \tau^2 \frac{\partial^2 b}{\partial x^2} 
+ \frac{\partial ^2 b}{\partial y ^2}  +
  2 \tau \frac{\partial^2 b}{\partial x \partial y} \right).
\]
As $b(x,0)=0$, the first two terms inside the parentheses are zero.
For the third term, we have, as mentioned just above Theorem \ref{sgcpthm1}, 
$\partial_y b = c(1+4a^2)$, so 
\[
\frac{\partial ^2 b}{\partial y ^2}  = \frac{\partial c}{\partial y}(1 + 4a^2)  + c\frac{\partial a^2}{\partial y},
\]
where $c(x,0) = \kappa(x)/2$ and $a(x,0) = \tau(x)/2$. We need an expression for $\partial_y c$, which we 
can obtain from the integrability condition $\dd \alpha + \alpha \wedge \alpha =0$
for the Maurer-Cartan form.  The $\mathfrak{k}$ part is 
\[
-\frac{\partial \bar U_\mathfrak{k}^t}{\partial z} - \frac{\partial U_\mathfrak{k}}{\partial \bar z} + [U_\mathfrak{p}, -\bar U_\mathfrak{p}^t ] = 0,
\]
from which we obtain 
\[
\frac{\partial c}{\partial y} = -\frac{\partial d}{\partial x} + b.
\]
From \eqref{dbexpressions} we have $d(x,0) = \kappa(x)\tau(x)/2$, and so 
$\partial_y c (x_0,0) = -\kappa^\prime(x_0) \tau(x_0)/2$.  
Hence, assuming $\kappa (x_0) = 0$,
\beqas
\eta^2 \mu(x_0) &=&-2\left(  \frac{\partial ^2 b}{\partial y ^2} (x_0,0)  + 
       2 \tau  \frac{\partial^2}{\partial x \partial y} b(x_0,0)) \right)\\
  &=& = -2 \left(  -\frac{1}{2} \kappa^\prime(x_0) \tau(x_0)\left(1+ \tau^2(x_0)\right)  +
			 \tau  \kappa^\prime(x_0)\left(1+\tau^2(x_0)     \right) \right) \\
		&=& \kappa^\prime (x_0) \tau (x_0) (1+\tau^2(x_0)).
 \eeqas
Finally, $\mu= -2b$ and  $\det(\hbox{Hess}(\mu) )=  4(b_{xx} b_{yy} - b_{xy}^2)$.  Thus, at $(x_0,0)$
\[
\det(\hbox{Hess}(\mu) )= -4 \left( \frac{\partial^2 b}{\partial_x \partial y}\right)^2 =
  -  \left(  \kappa^\prime(1+\tau^2)\right)^2.
\]
Now the solution has a degenerate singularity at $x_0$ if and only if $\kappa(x_0) = 0$.  Given this,
we have $\eta^2 \mu(x_0) \neq 0$ and $\det (\hbox{Hess}\, \mu (x_0)) <0$ if and only if $\kappa^\prime(x_0) \neq 0 \neq \tau(x_0)$. This proves the theorem.
	\end{proof}

	\subsection{The geometric Cauchy problem for non-regular singular sets}
	Now consider the geometric Cauchy problem for the case that $f_0: J \to \real^3$ is a real analytic map but not necessarily a regular curve.  In this case we cannot assume that $f_0(x)$ is unit speed. We assume that
	$\dd f$ has rank at least $1$, so that $f_x$ and $f_y$ cannot both vanish.  We again choose coordinates
	so that $f(x,0)=f_0(x)$.  Since we have already considered the case that $f_x$ is non-vanishing, the only
	new non-degenerate singularities are at points where $f_x$ vanishes. Therefore, we assume that $f_y$ 
	is non-vanishing, and  coordinates are chosen such that $f_y(x,0)$ has length $1$. Hence we can 
	choose a frame $F$ such that
	\[
	f_y = \Ad_F e_2, \quad \quad N= \Ad_F e_3.
	\]
	We can now do the computations as before, writing $U_\mathfrak{p}=(i a e_1 + (-1 + ib) e_2)/2$,
	so that $f_x = \Ad_F(-a e_1 -b e_2)$, and, for a singular curve, we must have $a(x,0)=0$, with
	non-degeneracy condition $a_y \neq 0$.
	Writing $U_\mathfrak{k} = (c+i d)e_3 /2$,  the harmonic map equations \eqref{harmeqns2} give
	$d=-bc$ and $\partial_y a = c(1+b^2)$, and we end up with, along $y=0$,
	\[
	U_{\mathfrak{p}}= \frac{-1+ib}{2} e_2, \quad U_\mathfrak{k}-\bar U_\mathfrak{k} = ce_3,
	\]
	with non-degeneracy condition $c \neq 0$.  From this, the following general solution to the non-degenerate
	singular geometric Cauchy problem follows:
	\begin{theorem}  \label{thmgeneral}
	All non-degenerate singularities on spherical frontals are locally obtained either by Theorem \ref{sgcpthm1}
	or by the DPW method with  potential
	\[
	\hat \eta =  \left(\frac{-1+i b(z)}{2}e_2 \lambda +  c(z) e_3   
	    + \frac{-1-i  b(z)}{2} e_2 \lambda^{-1}  \right) \dd z,
	\]
	where $b(z)$ and $c(z)$ are holomorphic extensions of real analytic functions $b$ and $c$,  $\real \to \real$
	satisfying $c(x) \neq 0$.  The singular curve $f_0(x)=f(x,0)$ satisfies
	$|f_0^\prime(x)| = |b(x)|$, and is:
	\begin{enumerate}
	\item a cone point if and only if $b \equiv 0$;
	\item
	diffeomorphic to a swallowtail at $(x_0,0)$ if and only if
	$b(x_0)=0$ and $b^\prime(x_0) \neq 0$;
	\item diffeomorphic to a cuspidal butterfly at $(x_0,0)$ if and only if 
	  $b(x_0)=b^\prime(x_0)=0$ and $b^{\prime \prime} (x_0) \neq 0$;
	\item  \label{thmgeneral_4}  a cuspidal edge if $b(x) \neq 0$.   In this case  we have
\[
b(x) = \frac{1}{\tau(x)}, \quad  \quad c(x) = \frac{\kappa(x)}{\tau(x)}.
\]
	\end{enumerate}
	\end{theorem}
\begin{proof}
The form of the potential $\hat \eta$ follows from the discussion above. By definition, a cone singularity is a 
non-degenerate singularity where the whole singular curve maps to a single point, which, in this case means
$b \equiv 0$.  Along the singular curve, the null direction is given by 
$\eta= \partial _x + b \partial_y$.
By a criterion in \cite{krsuy}, the singular curve is a cuspidal edge if and only if $\eta$ is transverse to the singular set, 
i.e.~if and only if $b \neq 0$, and a swallowtail at $x_0$ if and only if this transversality condition fails to first order at 
$x_0$, that is $b(x_0)=0$ and $b^\prime(x_0) \neq 0$. 
By Theorem 3.1 of \cite{fukuihasegawa}  (see also \cite{izumiyasaji}), a front is diffeomorphic to
a cuspidal butterfly at $p$ if and only if $\eta \mu (p) = \eta^2 \mu (p) = 0$ and 
$\eta^3 \mu (p) \neq 0$.  Here $\mu=a$ and it is simple to check that this condition
reduces to  $b(x_0)=b^\prime(x_0)=0$ and $b^{\prime \prime} (x_0) \neq 0$.

 The formulae for $b$ and $c$ at item \eqref{thmgeneral_4} follow by a direct computation
from $f_x = \Ad_F(-a e_1 -b e_2)$.
\end{proof}
	We remark that, in the case of cuspidal edges, if the torsion is non-vanishing, then this theorem gives an alternative
	representation to Theorem \ref{sgcpthm1}. Here the singular curve has speed $1/|\tau |$, 
	rather than unit speed.  Four examples are shown in Figure \ref{figure5}.
		
	\begin{figure}[ht]
\centering
$
\begin{array}{cccc}
\includegraphics[height=29mm]{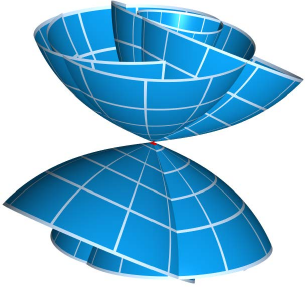}   \, & \,
\includegraphics[height=29mm]{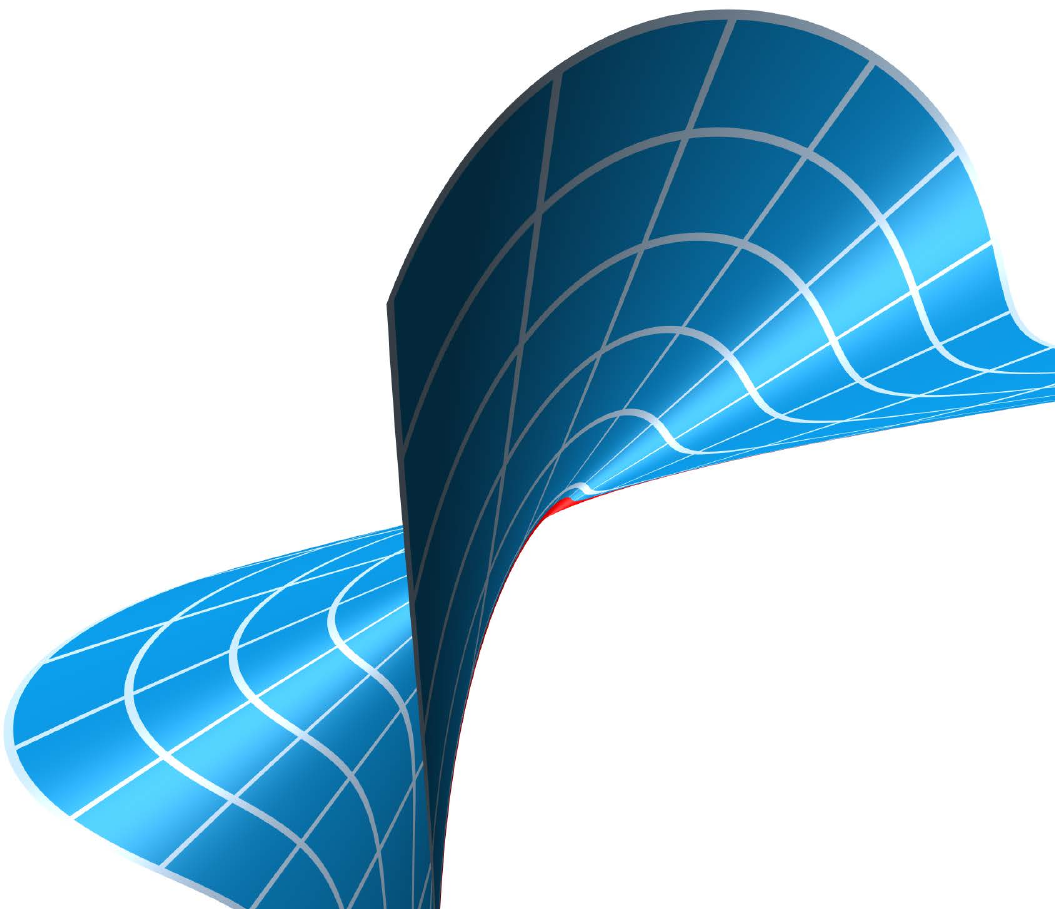}  \,& \, 
\includegraphics[height=29mm]{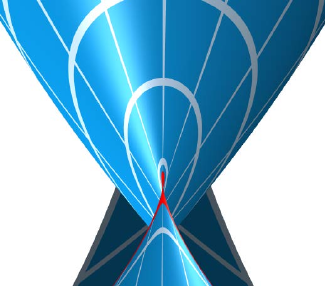}    \,& \,
\includegraphics[height=29mm]{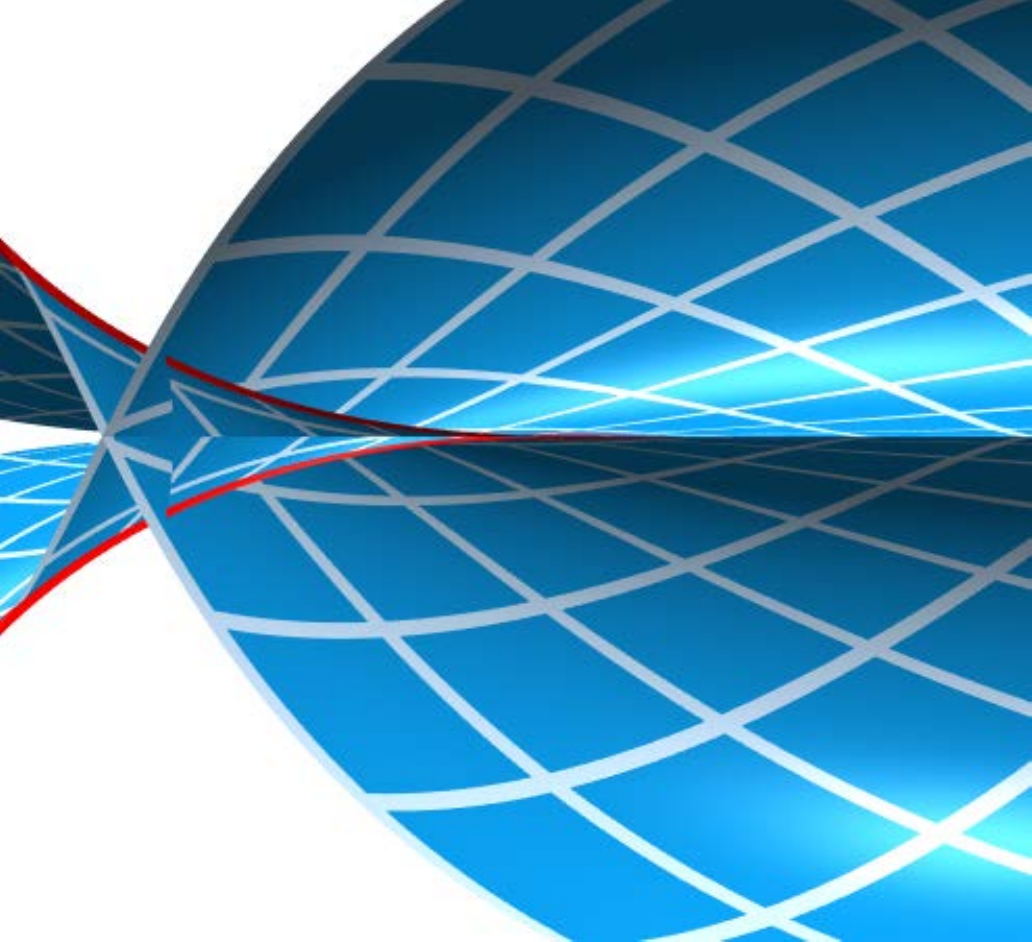}  \\
 c(x)=1+0.5\cos(x) &  (b(x),c(x))=(x^2,1) & (b(x), c(x))=(x,1) & (b(x),c(x))=(x,x)
\end{array}
$
\caption{Cone, cuspidal butterfly, swallowtail  and a degenerate singularity.}
\label{figure5}
\end{figure}

\subsection{The singular geometric Cauchy problem in terms of the unit normal}  \label{conessection}
G\'alvez et al.~\cite{ghm} studied spherical surfaces with isolated singularities. 
They use the term ``isolated'' from the point of view of the surface in $\real^3$,
rather than in the coordinate domain, so their notion of isolated includes a cone point.
Their results
are stated in terms of the unit normal. For example, they show (Theorem 12) that  an immersed
 conical singularity (i.e.~a cone point) is an embedded isolated singularity if and only if 
the unit normal along the singular curve is a regular strictly convex Jordan curve in $\SSS^2$,
and conversely such a curve in $\SSS^2$ gives rise to an embedded isolated singularity.
We can use Theorem \ref{thmgeneral} to compute the solution corresponding to such a curve
if we state it in terms of the normal:  with the same choices of coordinates and frame as above,
i.e.~$N=\Ad_F e_3$ and $f_y=\Ad_F e_2$, we have
\[
N_x = \Ad_F[U_\mathfrak{p}-\bar U_{\mathfrak{p}}^t, e_2] =  -\Ad_F e_1,
\]
so $x$ is also the arc-length parameter for $N$. Differentiating again we have
\[
N_{xx} = -\Ad_F[e_2-c e_3, e_1]  = -N - c N \times N_x.
\]
Hence $c(x)$ is the geodesic curvature function of the curve $N: \real \to \SSS^2$, and Theorem
\ref{thmgeneral} can be restated to the effect that, given an arbitrary analytic unit speed curve $N(x)$ in  $\SSS^2$ with non-vanishing geodesic curvature $c(x)$, any choice of analytic function $b(x)$ will generate a harmonic map $N(x,y)$ such that the associated spherical frontal has a non-degenerate singular curve along $y=0$. The choice $b\equiv 0$
gives the unique immersed conical singularity, and, according to \cite{ghm}, this is an embedded isolated singularity if and only if 
$N(x)$ is a closed curve.

	\begin{figure}[ht]
\centering
$
\begin{array}{ccc}
\includegraphics[height=28mm]{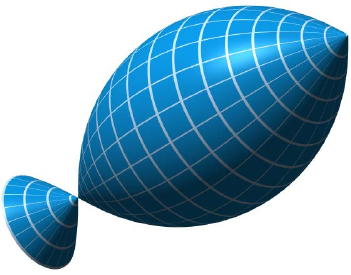}  \quad & \quad
\includegraphics[height=28mm]{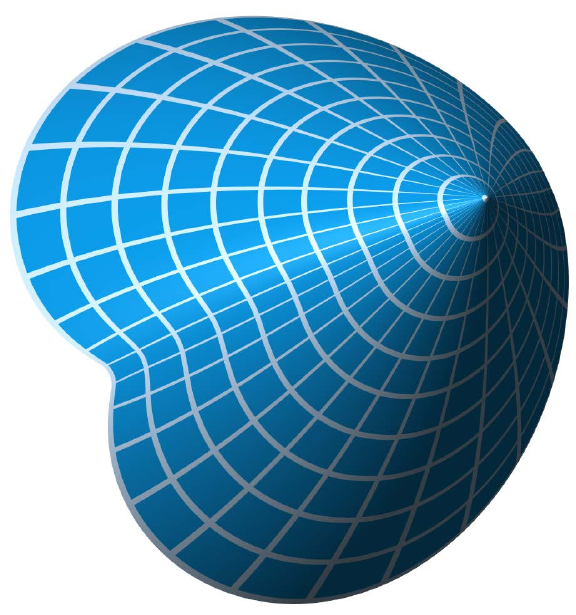}  \quad & \quad
\includegraphics[height=28mm]{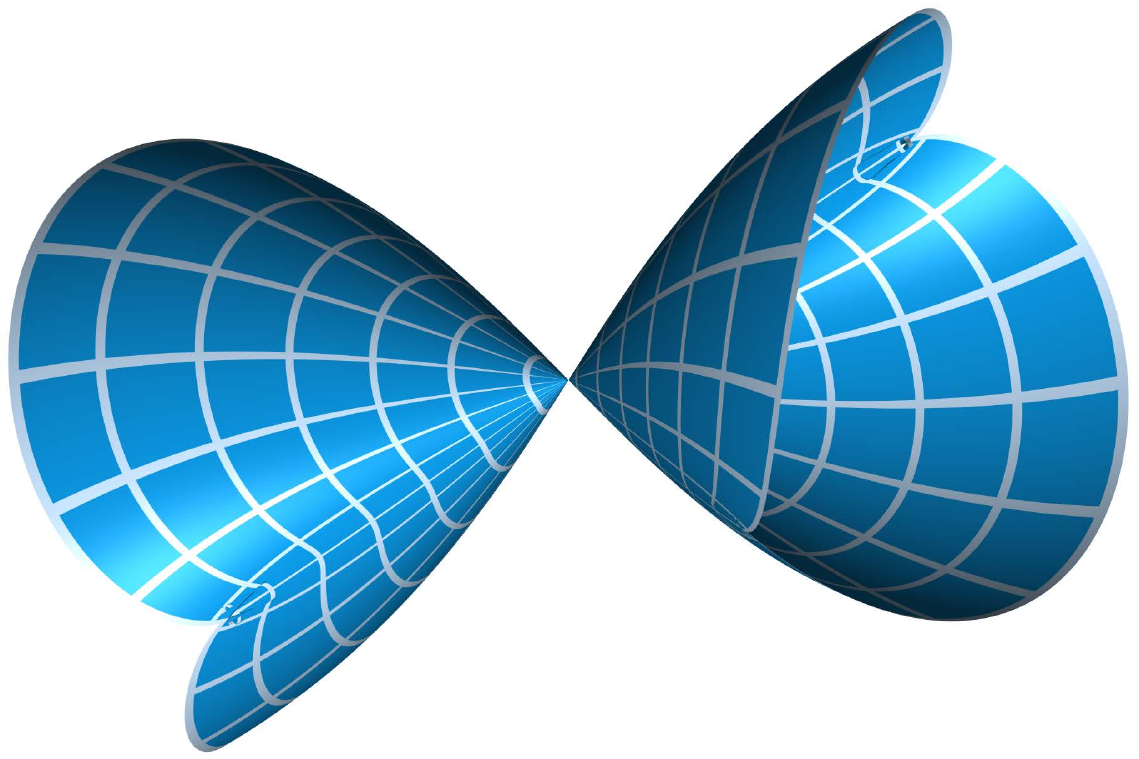}  
\end{array}
$
\caption{Examples of embedded cone singularities.}
\label{figurecones}
\end{figure}

For example, taking $N(s) = (R \cos(s/R), R \sin(s/R), \sqrt{1-R^2})$, a circle of radius $R$ in $\SSS^2$, the geodesic
curvature is $c= \sqrt{1-R^2}/R$, and the solution surface is a rotational ``peaked sphere''.  The peakiness is determined
by the radius $R$, i.e.~how close the normal at the cone point is to pointing towards the relevant pole of the sphere. For large $|c|$,
the solution approaches a chain of spheres, and for $c=0$ the solution degenerates to a straight line.  The case $c=1$,
corresponding to $N(s) = (1/\sqrt{2})( \cos(\sqrt{2}s), \sin(\sqrt{2}s), 1)$,
is shown in Figure \ref{figurecones} on the left. The embedded cone point shown to the right and middle was
obtained by perturbing this curve a small amount, retaining the closed property and the non-vanishing of the geodesic
curvature.

\section{Generating spherical surfaces from curves}  \label{generatingsection}
\subsection{Generating different surfaces from a single curve}
Given a real analytic space curve that is not a straight line, there are, of course, infinitely many different 
spherical surfaces containing this curve, determined by a choice of surface normal along the curve.
Moreover, one can say that we  have  three \emph{canonical} ways to produce a 
spherical surface from the curve: the unique spherical surfaces that contain the curve as a geodesic and as a cuspidal
edge, and the parallel spherical surface to the CMC $1/2$ surface that contains the curve as a geodesic.

Figure \ref{figure8} shows four different surfaces containing the plane curve defined by the curvature
function $\kappa(s)=1-s^4$:  the spherical surface and the unique constant negative curvature $-1$ surface (see \cite{singps})
that contain the curve as a cuspidal edge (which degenerates when $s=\pm1$), and the spherical surface and CMC $1/2$
surface that contain the curve as a geodesic.  
	\begin{figure}[ht]
\centering
$
\begin{array}{cc}
\includegraphics[height=25mm]{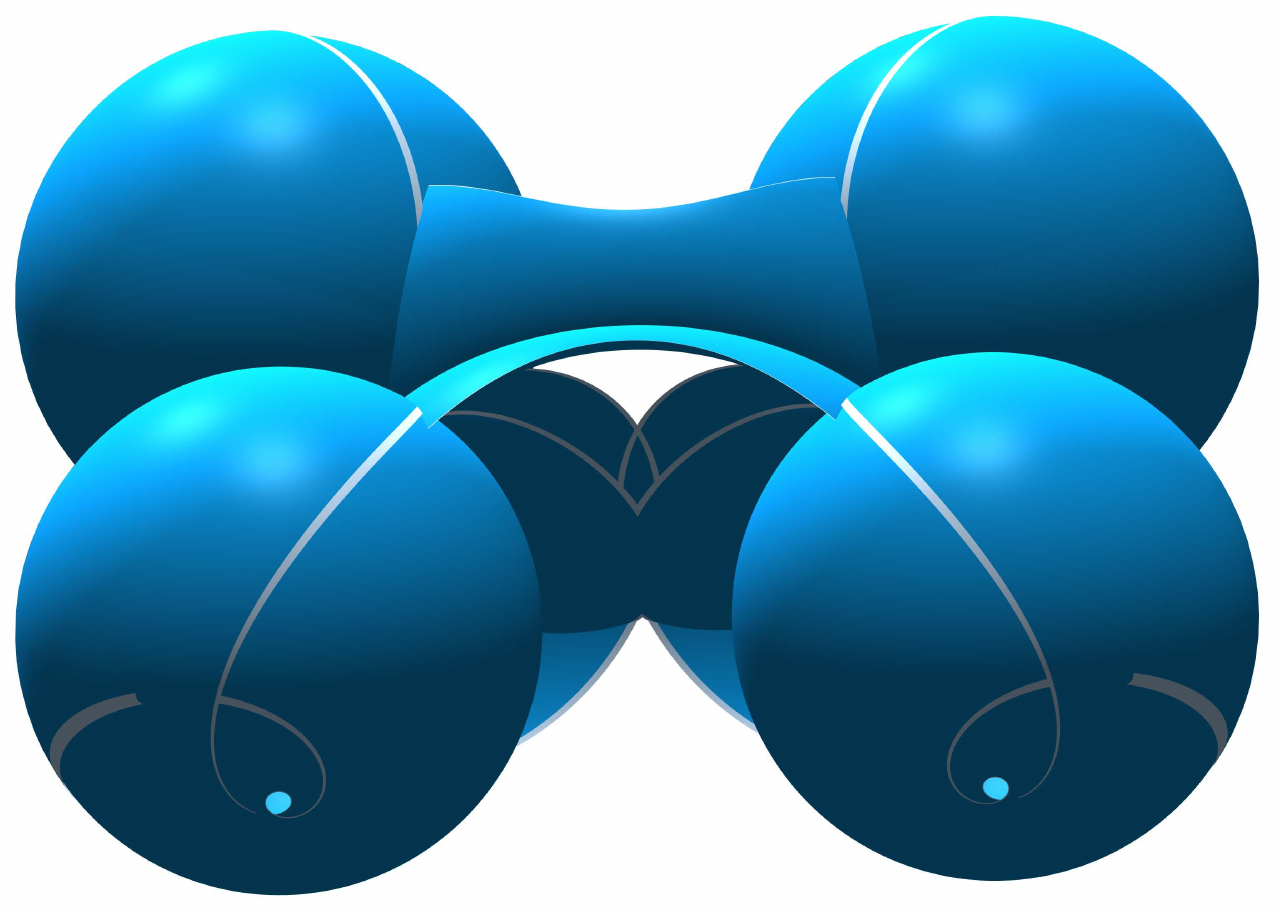}  &  \, \, 
\includegraphics[height=25mm]{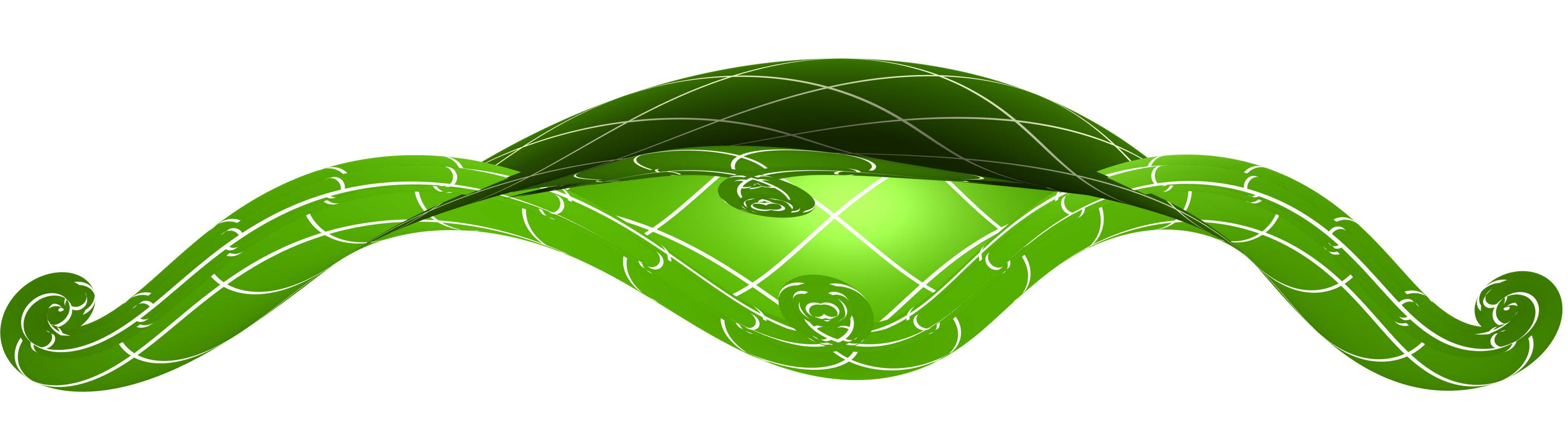}  
\end{array}
$
\vspace{3ex} \\
$
\begin{array}{ccc}
\includegraphics[height=28mm]{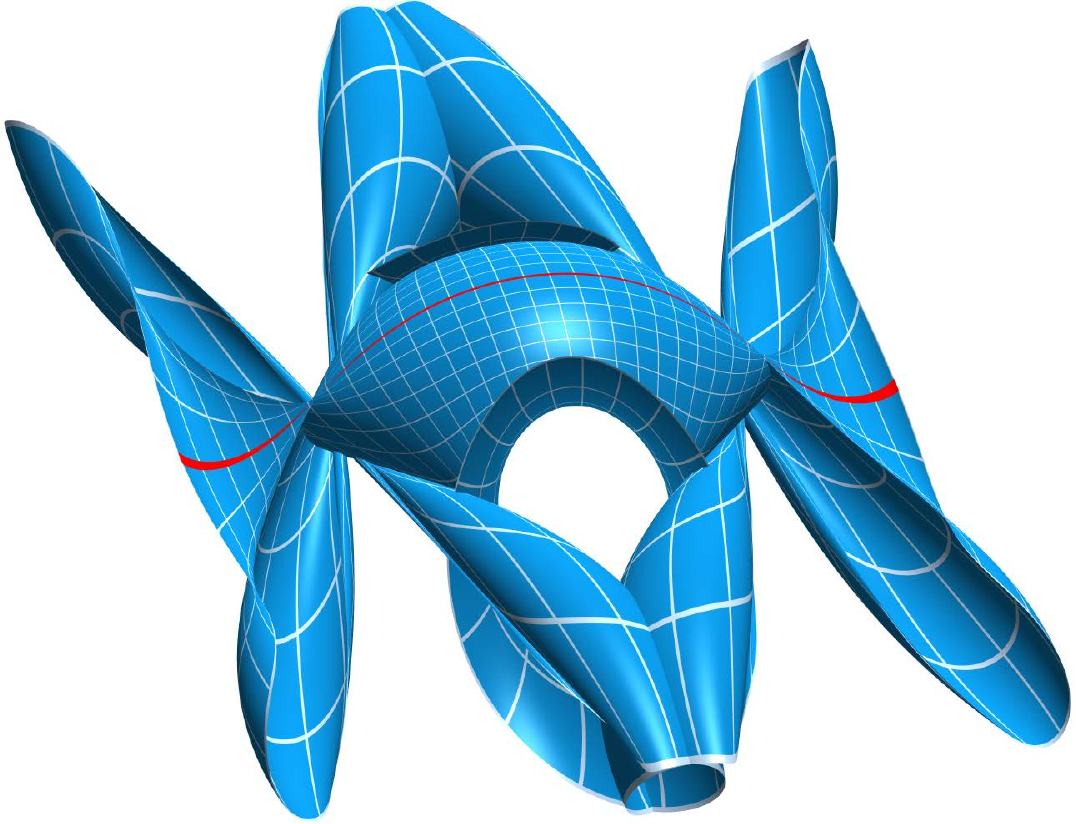} \quad   & \quad 
\includegraphics[height=28mm]{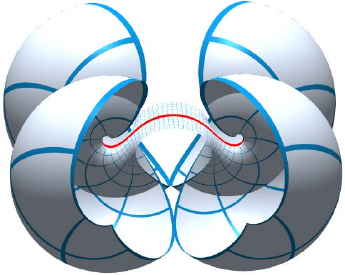}  \quad   & \quad 
\includegraphics[height=28mm]{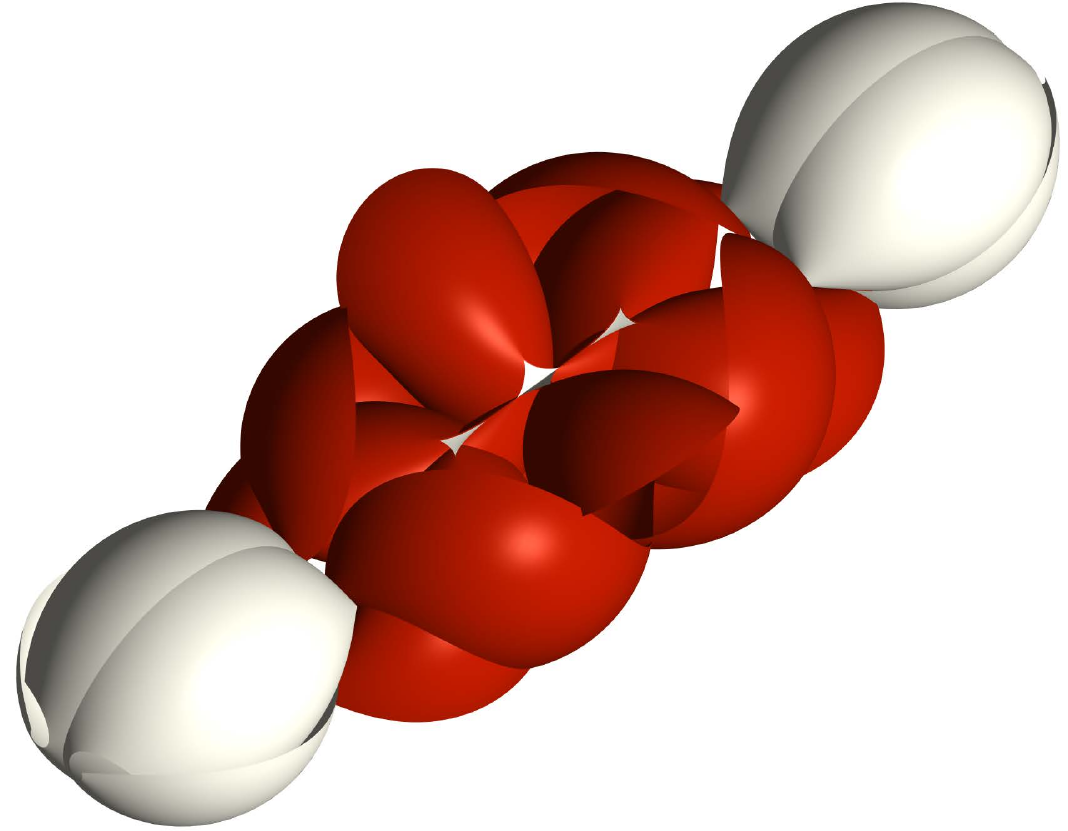} 
\end{array}
$
\caption{Surfaces generated by the plane curve with curvature function $\kappa(s)=1-s^4$. Top: spherical
and pseudospherical with the curve as a singular curve. Bottom:  spherical (left) and CMC $1/2$ (middle) 
with the curve as a geodesic.
Right: the \emph{parallel} spherical surface to the middle surface. See also Figure \ref{figure0}.}
\label{figure8}
\end{figure}

\subsection{Images of some surfaces discussed earlier}
In Figures \ref{figure9} and \ref{figure10} we display larger regions of some of the surfaces discussed earlier in the text.
The captions of the figures identify the relevant examples.

	\begin{figure}[ht]
\centering
$
\begin{array}{ccc}
\includegraphics[height=26mm]{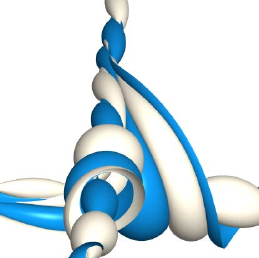}  \quad & \quad
\includegraphics[height=26mm]{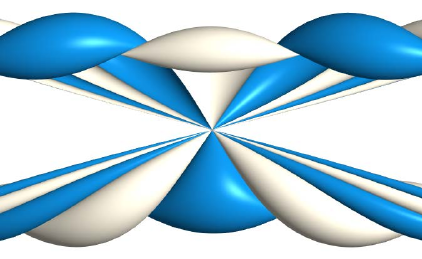}    \quad & \quad
\includegraphics[height=26mm]{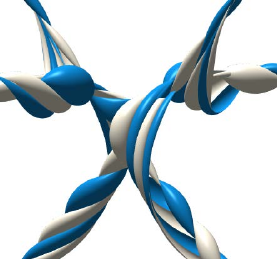} 
\end{array}
$
\vspace{3ex} \\
$
\begin{array}{ccc}
 \includegraphics[height=34mm]{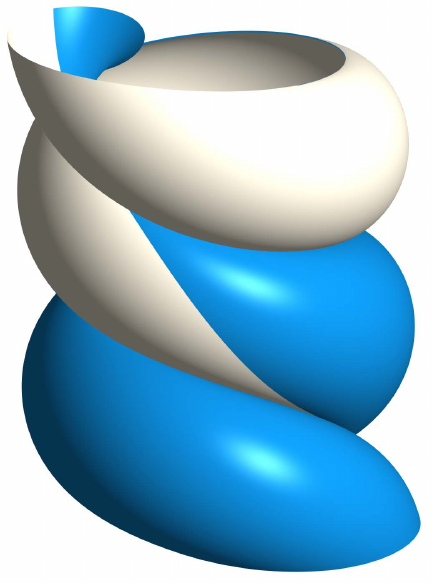} \quad & 
\includegraphics[height=34mm]{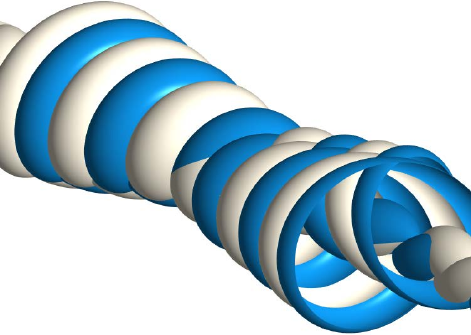}    \quad & 
\includegraphics[height=34mm]{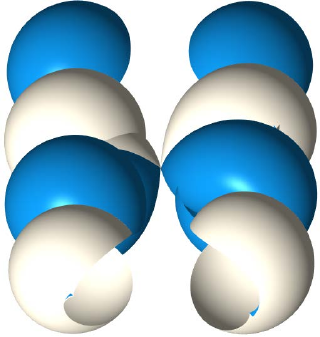}  
\end{array}
$
\caption{Top: the three spherical surfaces from Figure \ref{figure2}.
Bottom: The non-orientable spherical surface from Example \ref{nonorientableexample},
and the solution of the singular GCP for $\kappa(s)=s$, $\tau(s)=0$ (right).
}
\label{figure9}
\end{figure}

	\begin{figure}[ht]
\centering
$
\begin{array}{cc}
\includegraphics[height=40mm]{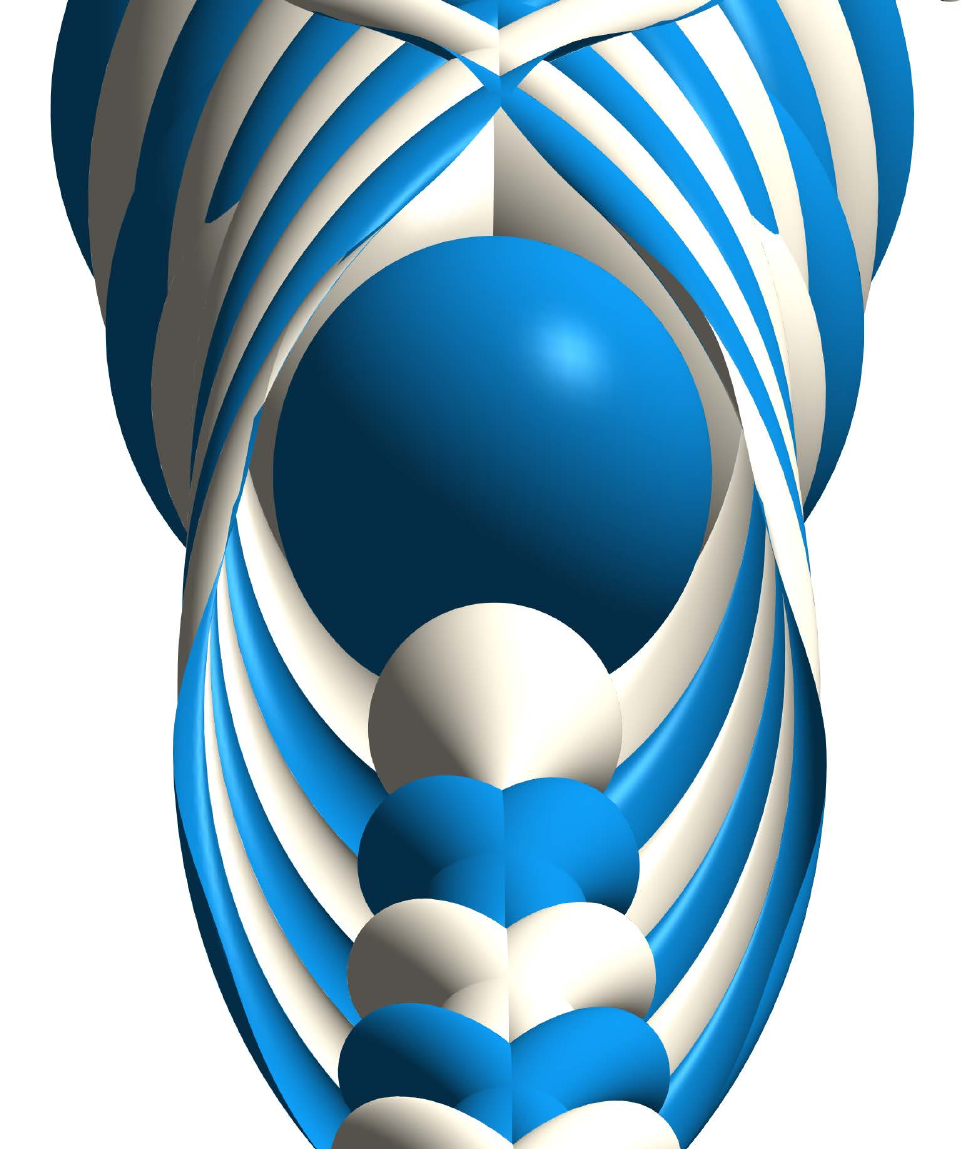} \quad  \quad & \quad 
\includegraphics[height=40mm]{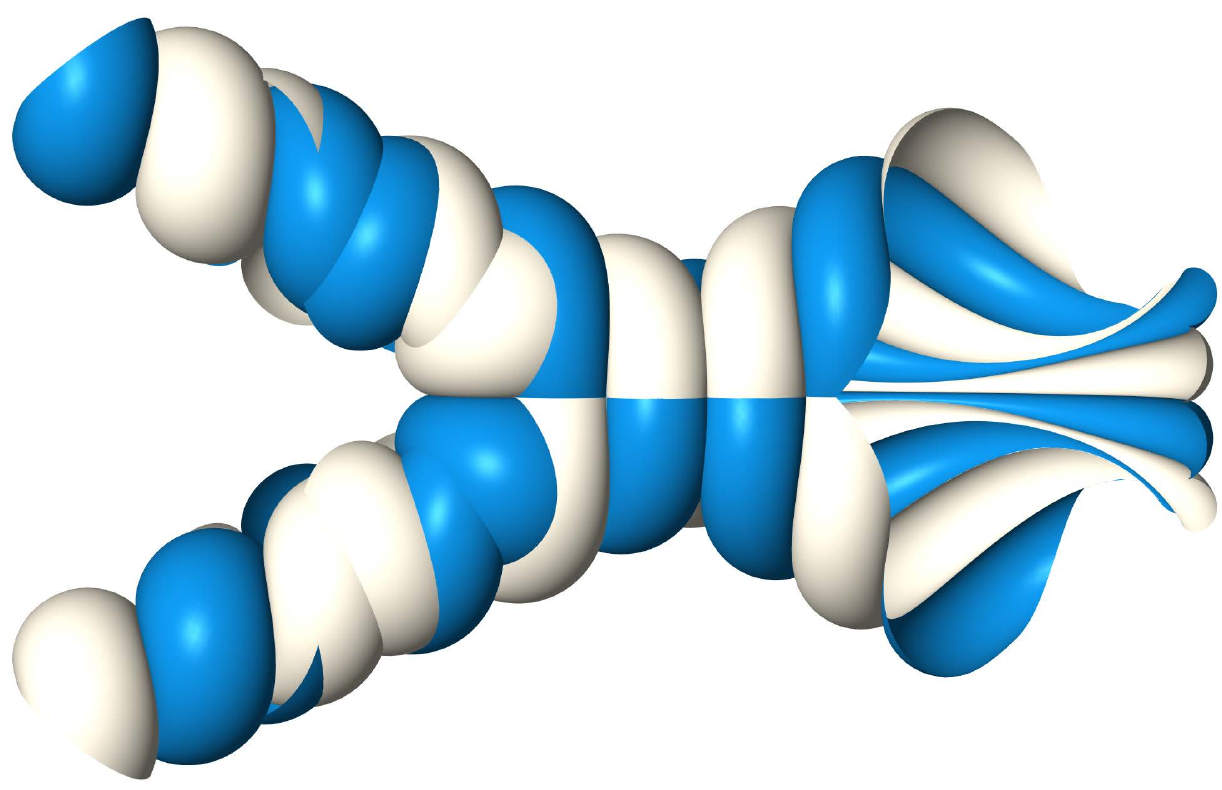}   
\end{array}
$
\caption{The spherical surfaces generated by Theorem \ref{thmgeneral} with $b(x)=x$, $c(x)=1$ (left), and
$b(x)=c(x)=x$ (right). See also Figure \ref{figure5}.}
\label{figure10}
\end{figure}

\subsection{External symmetries and closing experiments}
We saw in Section \ref{symmetriessection} how to obtain all spherical surfaces with a finite order
rotational symmetry about an axis that intersects the surface (at the fixed point on the surface).
 Another type of rotation is about an axis \emph{external} to the surface.  We can obtain these
using the solutions to the geometric Cauchy problem, choosing Cauchy data that has the required
symmetry.  We have already seen surfaces of revolution.  An example with a finite order
symmetry,  generated by an astroid curve is shown in Figure \ref{figureother}
to the left. 

It is of course simple to produce topological cylinders using the GCP, by choosing a closed initial
curve with periodic normal.  On the other hand, 
Figure \ref{figureother} shows, to the right, two examples
that are experimentally found to be ``closed'' in the $y$ direction.  They are generated 
by planar cuspidal edge curves, with, respectively, curvature functions $\kappa(t)=a\cosh(b t)$
and $\kappa(t)=c\sinh(d t)$, for experimentally found values of $a$, $b$, $c$ and $d$. 
Understanding under what conditions the surface closes up in the $y$-direction 
is an interesting problem that needs attention.

	\begin{figure}[ht]
\centering
$
\begin{array}{ccc}
\includegraphics[height=35mm]{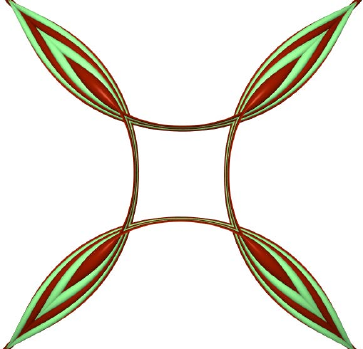}  \quad & \quad
\includegraphics[height=35mm]{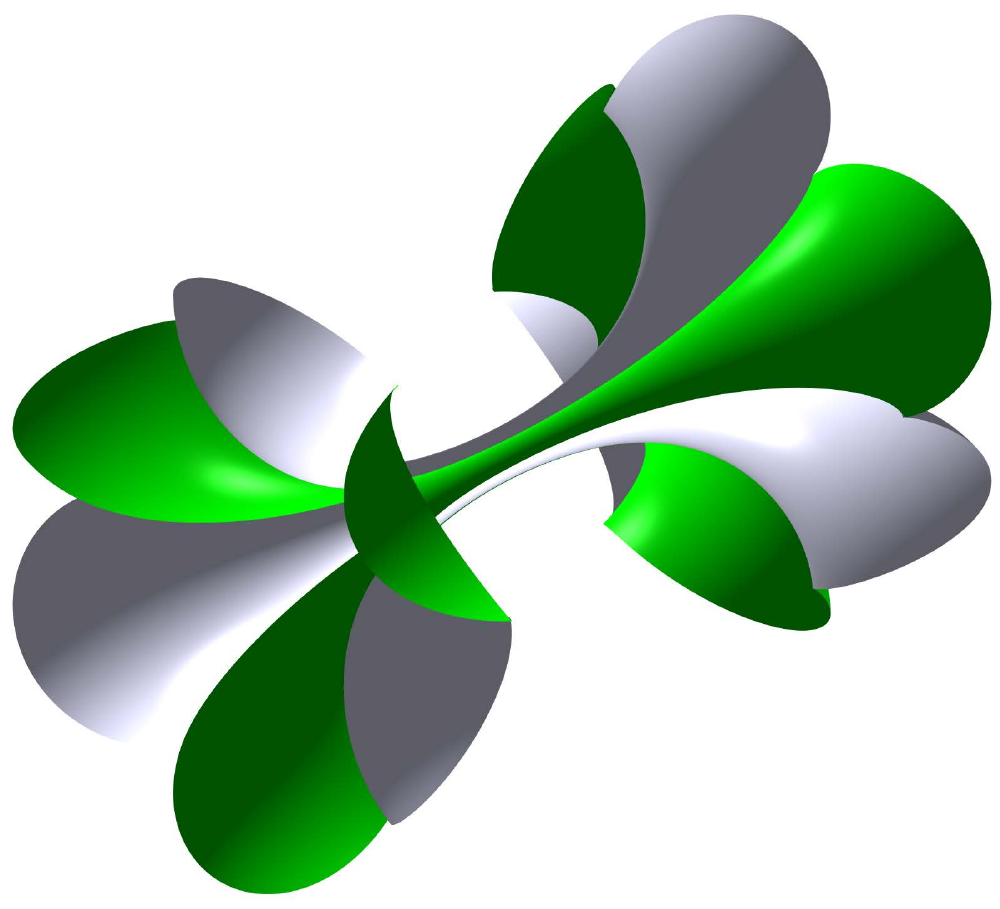}   \quad & \quad
\includegraphics[height=35mm]{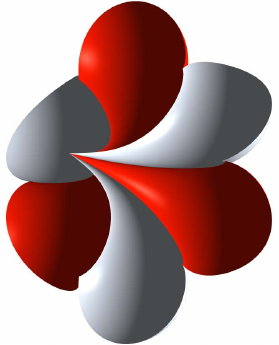}    
\end{array}
$
\caption{Left: example with external point of symmetry.  Middle, right: experimentally found ``cylindrical'' surfaces.
}
\label{figureother}
\end{figure}

\section{Conclusions and further questions}
 We have given two convenient tools for producing examples of spherical surfaces
 with prescribed geometric properties: 
 a method for producing all spherical surfaces with finite order rotational symmetries,
and solutions of the geometric Cauchy problems.  We have produced several
somewhat global-scale images in order to convey an idea of what spherical frontals in $\real^3$ look like.
They differ from constant negative curvature 
surfaces not only in their rounded appearance, but also because the cuspidal edges tend to be invisible as they protrude as ridges  on the ``inside'' rather than the ``outside'' of the surface. 
 For this reason, the coloring used in Figures \ref{figureB} - \ref{figureA},
 \ref{figure9}, \ref{figure10}, \ref{figureother} and \ref{figuretest},
 namely a binary color map that changes color when a cuspidal edge is crossed,
is quite informative for the global scale images. We conclude with some questions raised in this work:

	\begin{figure}[ht]
\centering
$
\begin{array}{cc}
\includegraphics[height=30mm]{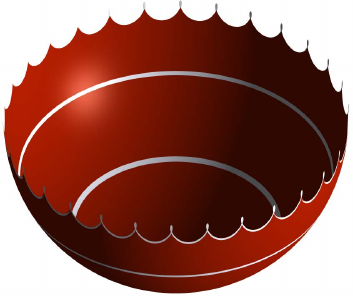}  \quad &  \quad
\includegraphics[height=30mm]{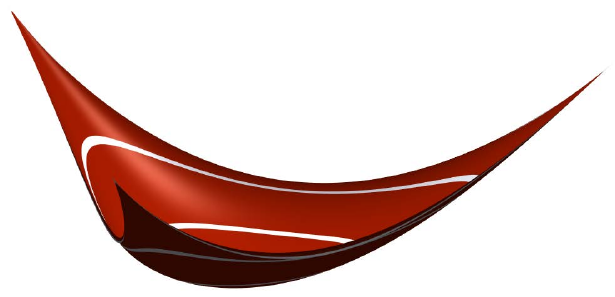}   \,\,\,
\includegraphics[height=30mm]{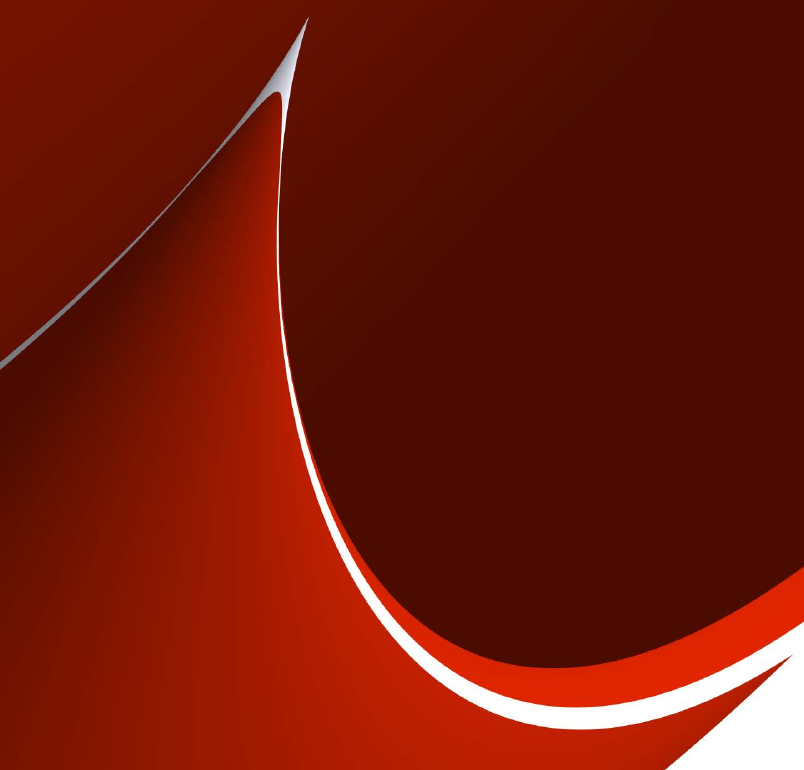}  \vspace{1ex}  \\
(1+z^{30}, \,  \,z^{28}) & (1+z^3, \, z) 
\end{array} 
$ 
\caption{Symmetric constant curvature topological discs, together with their natural boundaries.}
\label{figurediscs}
\end{figure}

	\subsection{Embedded  symmetric topological discs and their natural boundaries}
As we showed in Section \ref{symmetriessection},
if we consider a normalized potential of the form $(a(z), b(z)) = (1 + z^n, \, z^{n-2})$, for $n \geq 3$,
we obtain, on an open set around $z=0$, a constant curvature immersion, the image of
which has a rotational symmetry of order $n$.  We can consider the largest open set that is immersed,
and the boundary of this set as the ``natural boundary'' of an immersed spherical surface. 
This boundary does not have any isolated points because one can easily strengthen 
the statement of Lemma \ref{regularitylemma} to show that the surface has a branch point at \emph{any} point $z$
if and only if $a(z)=b(z)=0$.  Thus, this boundary is a curve, or a collection of curves.  
Two examples are shown in Figure \ref{figurediscs}.  Having computed the solution for several different values of
$n$, it seems clear that one always obtains, as maximal immersed neighbourhood, an
embedded topological disc. The boundary of the image is a closed
curve with $n$ cusp points, and, of course, has an order $n$ rotational symmetry. The cusp points
correspond to swallowtail singularities.  This curve does not correspond to a round circle
in the parameter domain, as can already be seen
in the close-up on the right in Figure  \ref{figurediscs}, where one can see the image of a constant 
$r$ curve, in $(r,\theta)$ polar coordinates, crossing
the cuspidal edge. 

Investigating further, we compute now the solution for $(a(z), b(z)) = (1 + z^3,  c \, z)$, for various
values of $c$, displayed  in Figure \ref{figuretest}.  We also plot the domain for each surface, underneath,
and both plots are coloured according to the sign of the Gauss curvature of the parallel CMC surface.
The singular curves are where the colour changes.   In the last two examples, the maximal immersed
surface around the central point has non-trivial topology.   These observations raise the question of
whether one can predict the type of singularities, and the topology of maximal immersed subdomains
for a spherical surface with a given (perhaps special) normalized potential. 

\begin{figure}[ht]
\centering
$
\begin{array}{cccc}
\includegraphics[height=24mm]{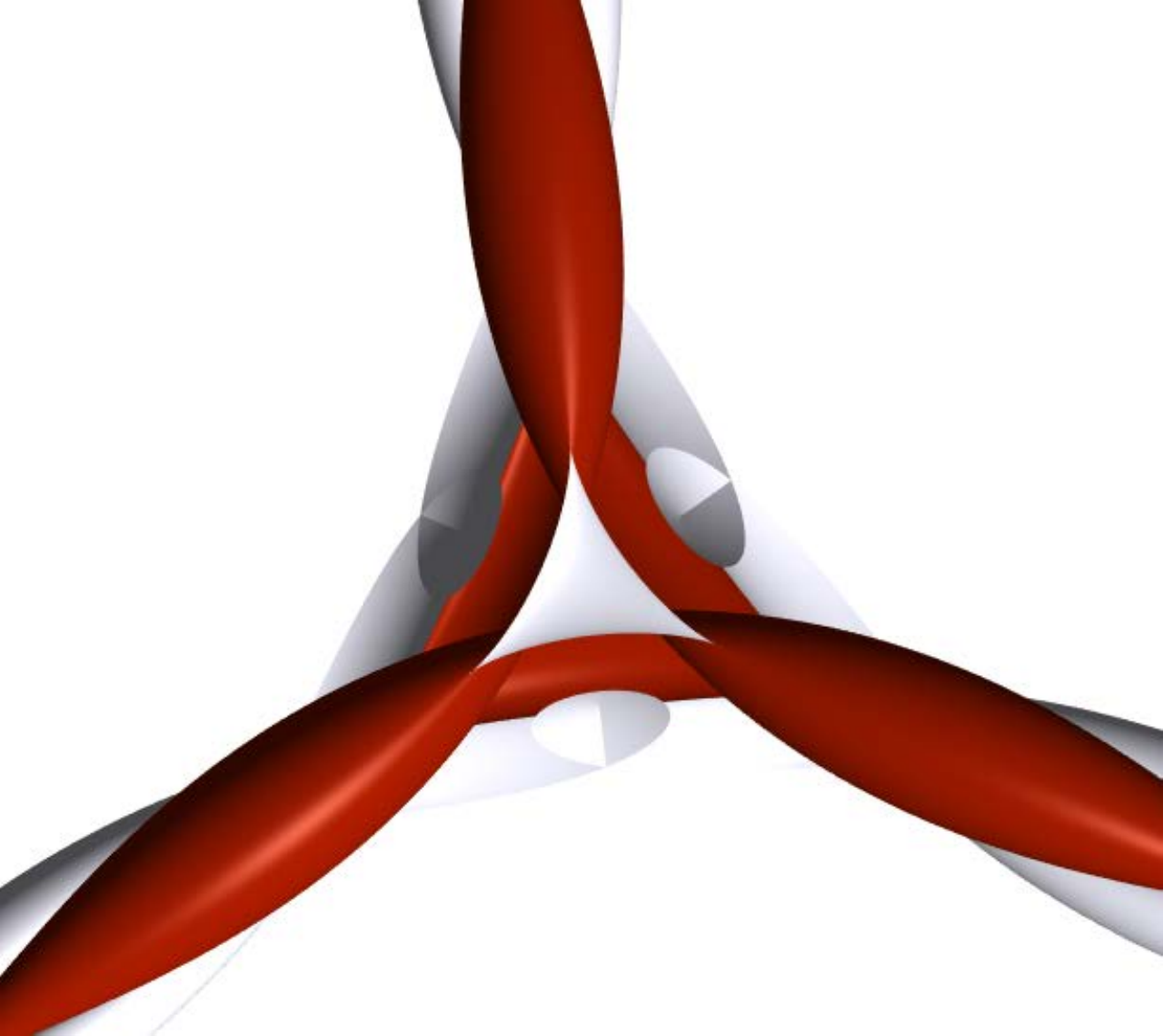}  \quad &  \quad
\includegraphics[height=24mm]{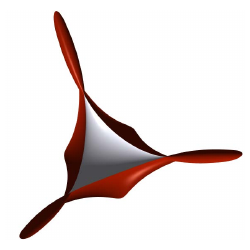}   \quad &  \quad
\includegraphics[height=24mm]{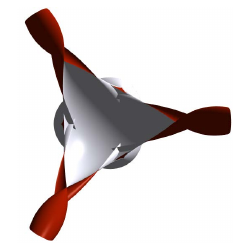}   \quad &  \quad
\includegraphics[height=24mm]{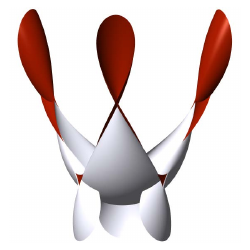}     \vspace{1ex} \\
\includegraphics[height=24mm]{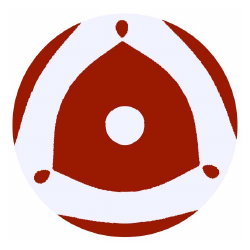}    \quad &  \quad
\includegraphics[height=24mm]{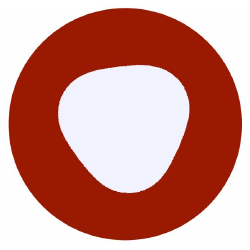} \quad & \quad
\includegraphics[height=24mm]{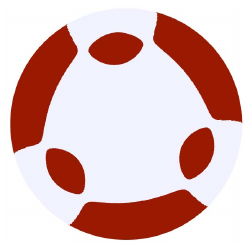}   \quad &  \quad
\includegraphics[height=24mm]{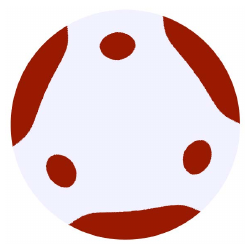}    \vspace{1ex}   \\
  (1+z^3, 5 \, z)  & (1 + z^3, \, z)  &  (1+ z^3, z/3)  &  (1+ z^3, z/4)
\end{array}
$
\caption{Topology of the maximal immersed set for some examples.}
\label{figuretest}
\end{figure}

\subsection{Spherical surfaces contained in bounded subsets of $\real^3$}
In related work on constant \emph{negative} curvature surfaces \cite{singps}, one observes
 the existence of  frontals that appear to be completely contained in a bounded region
of $\real^3$.  The pseudospherical surface generated by the planar singular curve with curvature function
$1-s^4$ in Figure \ref{figure8} is one such example, and other planar curves which satisfy
$|\kappa| \to \infty$ as $|t| \to \infty$ appear to behave the same way: namely, the surface spirals inwards in 
both the $x$ and $y$ directions, with many singularities.  At the experimental level, it
seemed to be quite easy to find pseudospherical frontals that are contained in a bounded set, and this 
sometimes makes computing
the solutions more rewarding.  For the case of spherical surfaces, if one takes the surface parallel to a CMC torus 
(and CMC tori have been extensively investigated), then this will also be a torus, with singularities, and hence
bounded. However tori are few and far between, and it did not seem to be easy to find initial curves
for spherical surfaces that generate a bounded solution.

\subsection{The bifurcations of one parameter families}
Ishikawa and Machida \cite{ishimach} proved that the generic (or stable)
singularities for constant curvature surfaces are cuspidal edges and swallowtails.  A singularity 
(speaking at the level of map germs) is  generic within a class of surfaces if ``nearby'' surfaces of
the same class necessarily have the same kind of singularity.  It is intuitive what this means if one
considers  Theorem \ref{thmgeneral}, which locally generates all 
non-degenerate solutions to the singular geometric Cauchy problem from  pairs of real-analytic
functions $(b(t),c(t))$.  The non-degeneracy condition is $c(t) \neq 0$. The solution is a cuspidal
edge around a point $(x_0,0)$ where $b(x_0) \neq 0$. Perturbing the data slightly, nearby
pairs of functions will also be non-vanishing at $x_0$, so a cuspidal edge is generic. 
Similarly, the surface is a swallowtail at a point where $b$ vanishes to first order, and this
also is a generic condition.

After the generic singularities, the next most important singularities are the \emph{bifurcations}
in generic one parameter families of surfaces of the relevant class.  The bifurcations for generic
families of fronts are classified in \cite{arnoldetal}.  They are known as cuspidal lips, cuspidal beaks,
cuspidal butterflies and $3$-dimensional $D_4^+$ and $D_4^-$ singularities (Figure \ref{figurestandard}).
They are respectively given around the point $(0,0)$  by the map germs
 $(u,v) \mapsto (3u^4+2u^2v^2, u^3+uv^2,v,)$,   $(u,v) \mapsto (3u^4-2u^2v^2, u^3-uv^2,v)$, 
 $(u,v) \mapsto (4u^5+u^2v, 5u^4+2uv,v)$,   $(u,v) \mapsto  (uv, u^2 + 3v^2, u^2v  + v^3)$, 
 $(u,v) \mapsto  (uv, u^2 - 3v^2, u^2v  - v^3)$.

	\begin{figure}[ht]
\centering
$
\begin{array}{ccccc}
\includegraphics[height=24mm]{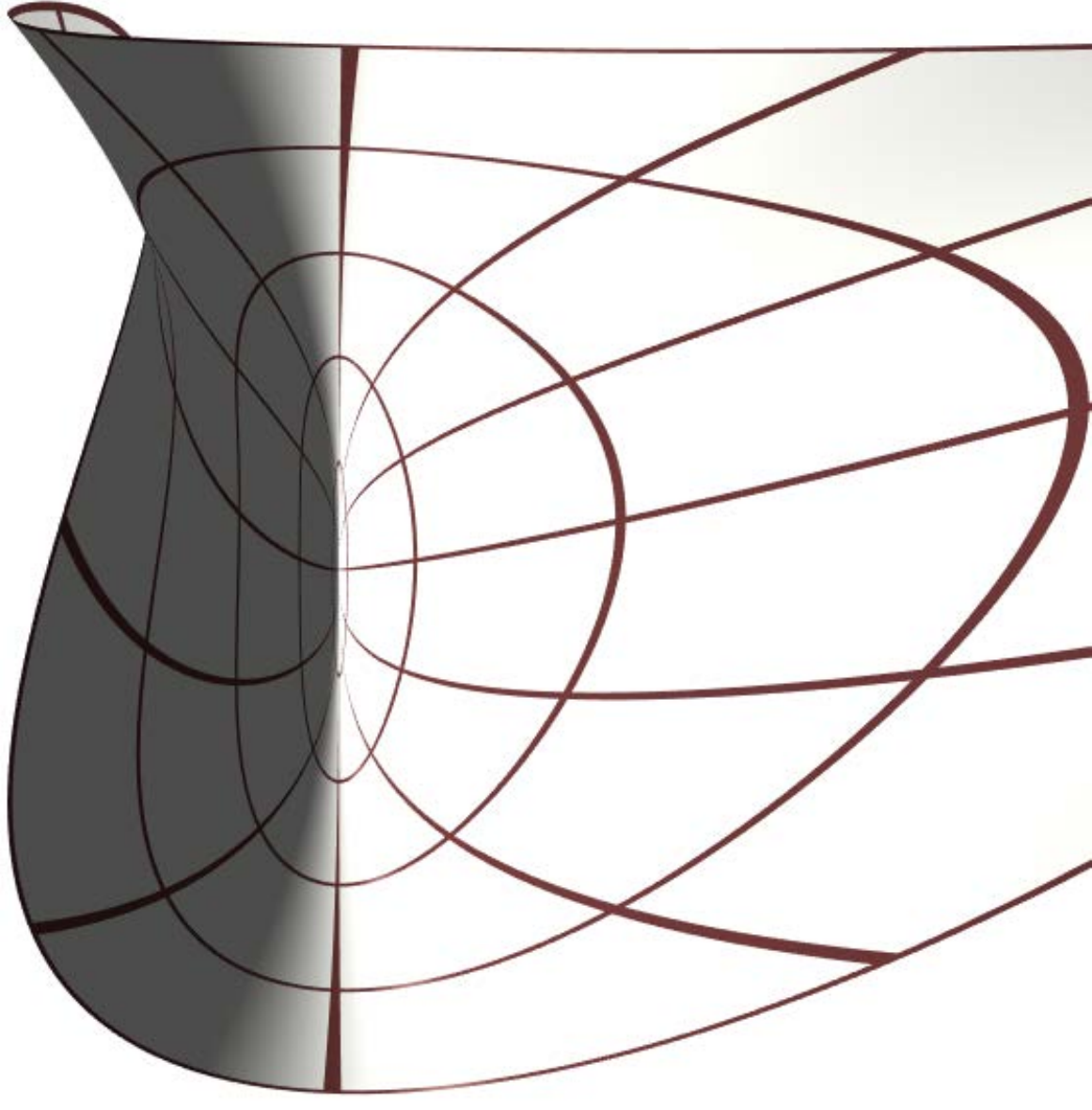}   \, \, & \,\,\,
\includegraphics[height=24mm]{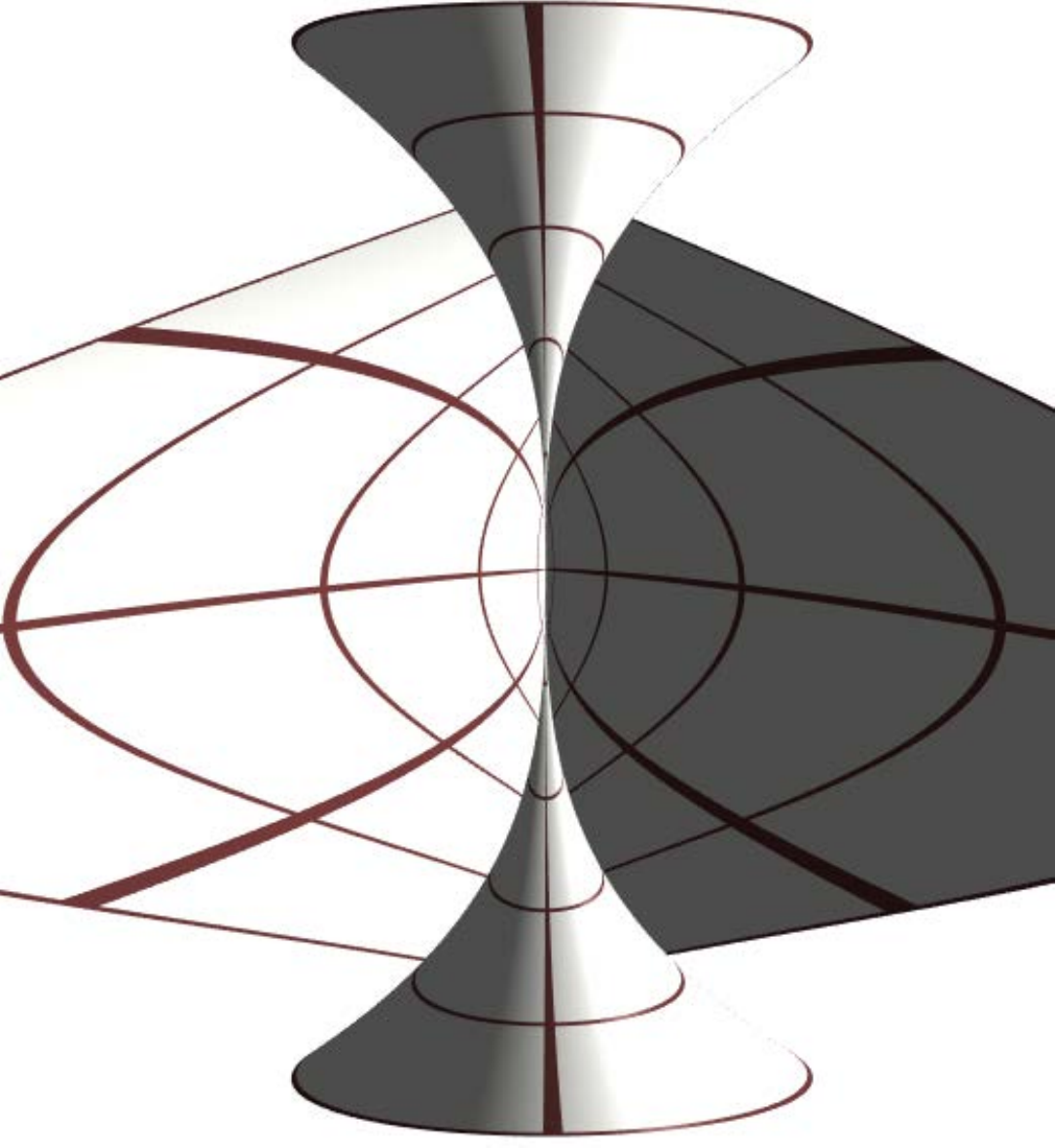}  \, \, & \,\,\,
\includegraphics[height=24mm]{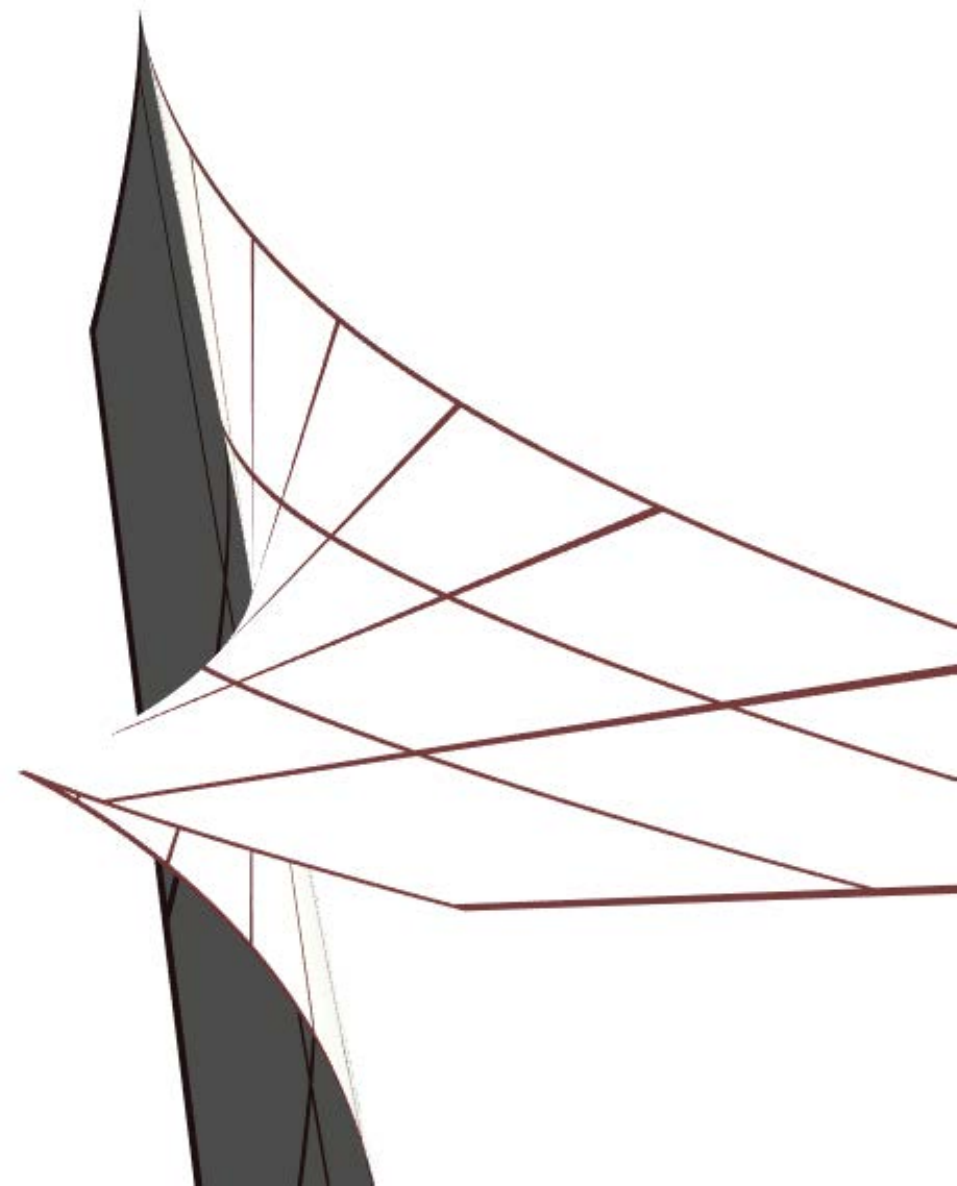}    \,& \,\,
\includegraphics[height=24mm]{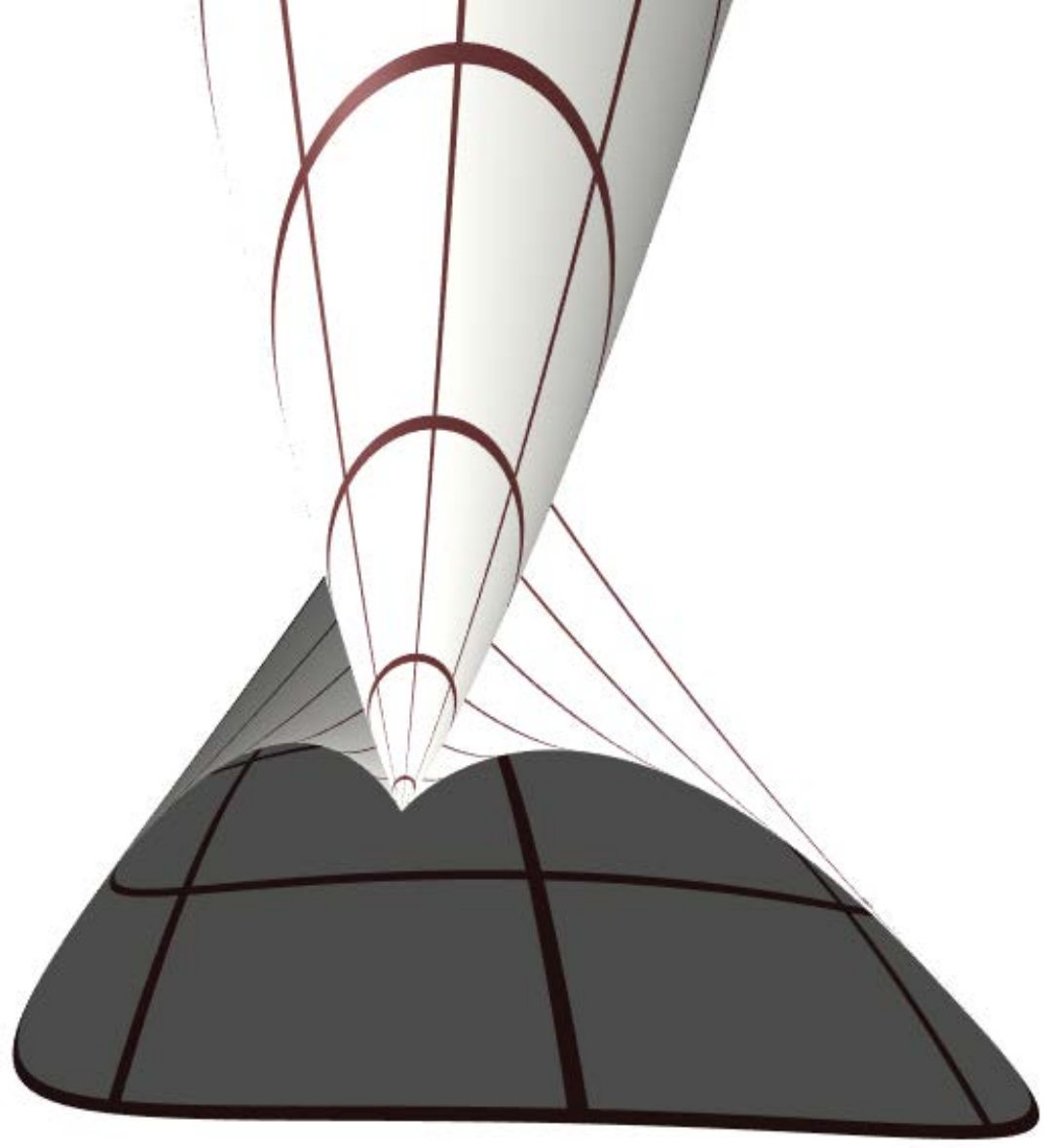} \,& \,
\includegraphics[height=24mm]{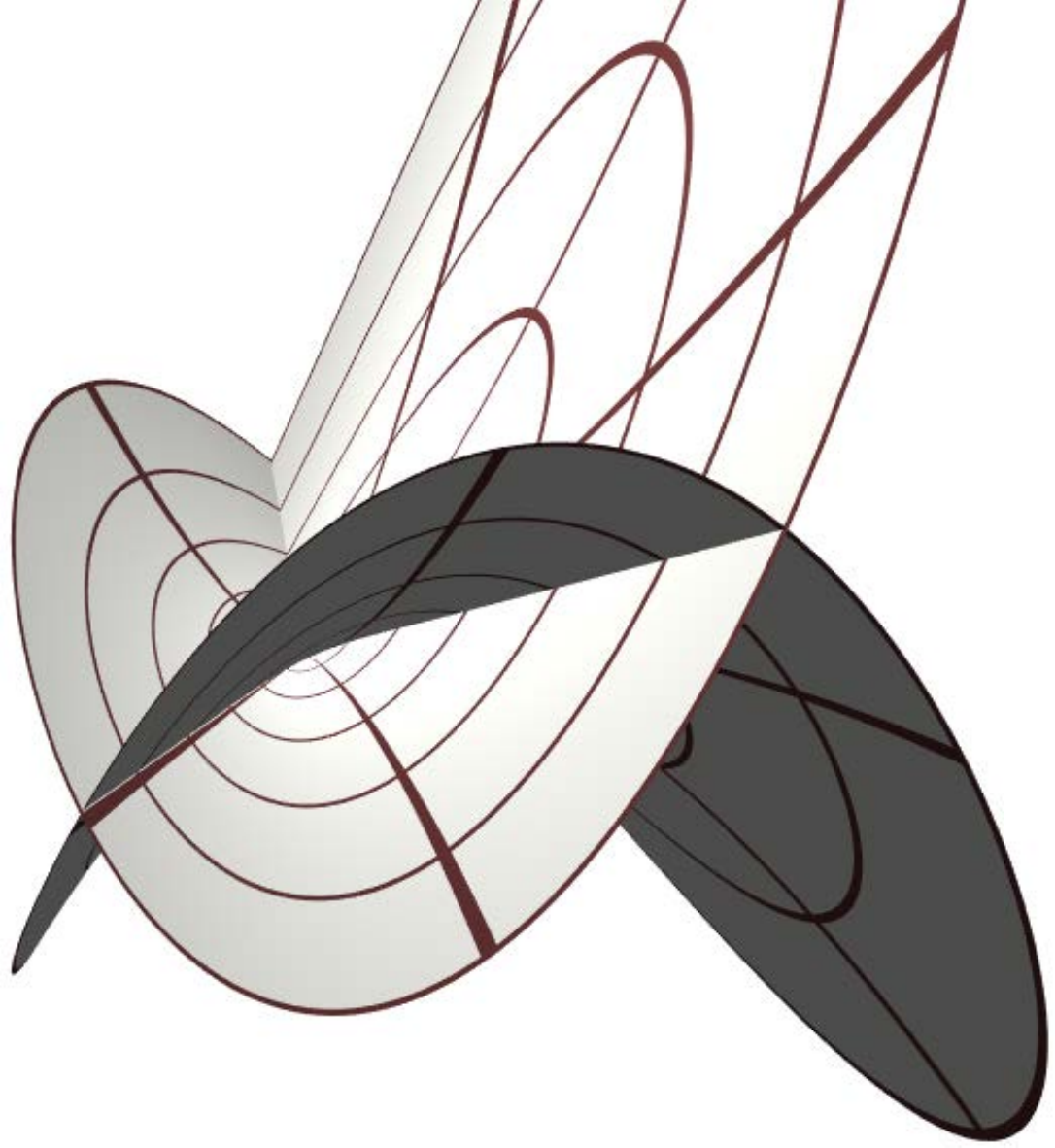}  \\
\end{array}
$
\caption{In order: standard cuspidal lips, cuspidal beaks, cuspidal butterfly, and $3$-dimensional $D_4^+$ and $D_4^-$ singularities. Only the second two can occur on spherical frontals.}
\label{figurestandard}
\end{figure}

The cuspidal lips has an isolated singularity at $(0,0)$, but the derivative has rank $1$ at this
point.  Hence this singularity cannot occur on a spherical frontal, by Lemma 
\ref{regularitylemma}.  The derivative of the $D_4^+$ singularity is
zero at $(0,0)$, but the singularity is not isolated, hence this singularity cannot occur either, by the
same lemma.  The $D_4^-$ singularity has the property that its derivative vanishes at $(0,0)$ but 
the derivative of its unit normal does not, and this is impossible for spherical frontals.
Hence, from this list, only cuspidal beaks and butterflies  occur on spherical frontals.

If we consider the generic geometric Cauchy data for the singular geometric
Cauchy problems  in Theorems \ref{thmgeneral} and  \ref{sgcpthm1}, and the singularities
one obtains when the data fails to first order to give a cuspidal edge or a swallowtail,
it seems plausible that the bifurcations of generic one parameter families of spherical fronts 
are cuspidal beaks,  cuspidal butterflies and the singularity given by Theorem  \ref{thmgeneral} with
data $(b(t),c(t))=(t,t)$.  However, this reasoning is based only on considering
families of the special potentials used to solve the singular GCP here, and therefore needs
some additional argument.

Finally, a remark on the experimental aspects of this work: 
some of the results of this article are proved for the purpose of providing tools
for computing concrete examples of spherical surfaces. Conversely, some of these
results have been proved \emph{after} observing them to be likely to be true through 
computations.  Examples are the criteria for cuspidal beaks and cuspidal butterflies:
these were observed experimentally and the singularities visually identified before
the analysis was done.
The code used to compute these solutions  was written in Matlab and can be found, 
at the time of writing, at:
\href{http://davidbrander.org/software.html}{http://davidbrander.org/software.html}.


\begin{appendices}

\section{The geometric Cauchy problem for CMC surfaces}  \label{appendix}
	The geometric Cauchy problem for CMC surfaces is solved in \cite{bjorling}. However, explicit formulae for the boundary potentials
	in terms of the Cauchy data are not given in that work. The method described there involves first finding an $SU(2)$ frame
	along the curve, a burden if one would like to compute many different examples.  We therefore derive here the boundary
	potentials directly in terms of the curvatures and torsion of the curve, to give a simple
	and convenient means for producing examples of CMC surfaces.
	
	From the discussion in Section \ref{parallelsection}, if $f$ is a spherical frontal with unit normal $N$, then $g=f-N$ is the parallel surface
	with constant mean curvature $H=1/2$, and satisfies
	$g_z = f_z -N_z = i N \times N_z -N_z$, i.e.~we consider the defining equation
	\beq  \label{CMCeqn}
	g_z = i N \times N_z -N_z.
	\eeq
	Since the integrability condition for this equation is again the harmonic map equation
	$ N \times N_{z \bar z}=0$, we can use the same loop group representation to obtain CMC surfaces, subtracting the unit normal
	$N$ from the previous Sym formula, to obtain: $\mathcal{SB}(\hat F) := \mathcal{S}(\hat F) - \Ad_F e_3$.
	
As with spherical surfaces, we assume given a real analytic curve $g_0(x) = g(x,0)$, and a prescribed normal 
$N_0(x)=N(x,0)$, and then derive the boundary potentials for the CMC $1/2$ surface $g$ satisfying this
Cauchy data in terms of the curvature and torsion of $g_0$.  The potentials are different from the spherical case
because the equation \eqref{CMCeqn} is different from \eqref{sphericalsurface}.  Again choosing a frame with
$N= \Ad_F e_3$, equation \eqref{CMCeqn} gives us
\[
\Ad_{F^{-1}} g_z = i U_\mathfrak{p} - [U_\mathfrak{p}, e_3], \quad \quad
\Ad_{F^{-1}} g_{\bar z} = i \bar U_\mathfrak{p} + [\bar U_\mathfrak{p}, e_3].
\]
We can assume that $g$ is conformally immersed, and can even assume that coordinates are chosen such that the
conformal factor is $1$ along $y=0$, so that $x$ is the arc-length parameter along this curve.
 Thus, we assume the frame is chosen satisfying
\[
g_x = \nu \Ad_F e_1, \quad \quad g_y = \nu \Ad_F e_2, \quad \quad \nu(x,0)=1,
\]
where $\nu: \Omega \to (0,\infty) \subset \real$ is real analytic.  Substituting this into the above equations,
we conclude that, along $y=0$,
\[
U_\mathfrak{p} = \frac{1}{2}(a+b i) e_1 + \frac{1}{2}(-b-1+ai)e_2,
\]
where $a$ and $b$ are real-valued functions.  
Next write $U_\mathfrak{k}-\bar U^t_\mathfrak{k} = c e_3$.
Differentiating $g_x = \nu \Ad_F e_1$, with $\nu(x,0)=1$, we obtain, along
$y=0$,
\[
g_{xx} = \Ad_F[U-\bar U^t, e_1]  = c \Ad_F e_2 + (b+1) \Ad_F e_3.
\]
Since $x$ is the arc-length parameter along this curve, we also have $g_{xx} = \kappa_g \Ad_F e_2 + \kappa_n \Ad_F e_3$,
where $\kappa_g$ and $\kappa_n$ are the geodesic and normal curvatures of $g_0$. Thus
$c=\kappa_g$ and $b=\kappa_n -1$.  To find $a$, differentiate $N= \Ad_F e_3$ to obtain
\[
N_x = \Ad_F[U-\bar U^t, e_3] = -a \Ad_F e_2 -(b+1) \Ad_F e_1,
\]
so $a=-\langle N_x, \Ad_F e_2\rangle = \langle N_x , g_x  \times N \rangle$.
Substituting $U_\mathfrak{p}$ and $U_\mathfrak{k}-\bar U_\mathfrak{k}^t$ into \eqref{F0mcform}, we obtain:
\begin{theorem}  \label{bjorlinthm}
Let $\gamma: J \to \real^3$ be a regular arc-length parameterized real analytic curve and suppose given an
 analytic vector field  $N:J  \to \SSS^2 \subset \real^3$  , satisfying $\langle \gamma^\prime(s), N(s) \rangle =0$.
Set 
\[
\mu := \langle \gamma^\prime\times N , N^\prime \rangle, \quad \kappa_g := \langle \gamma^{\prime \prime}, N \times \gamma^\prime \rangle, \quad
\kappa_n := \langle \gamma^{\prime \prime}, N \rangle.
\]
Then the unique solution to the geometric Cauchy problem for $\gamma$ and $N$ is given by the DPW method with
the holomorphic potential, given by the analytic extension of:
\beqas
 \hat \eta &=& \left( \frac{1}{2} ((\mu - (\kappa_n -1)i ) e_1 + (-\kappa_n - \mu i) e_2) \lambda + 
     \kappa_g e_3   \right.  \\
  &&  \left. + \frac{1}{2}((\mu+(\kappa_n-1)i)e_1 + (-\kappa_n + \mu i) e_2 ) \lambda^{-1} \right) \dd s.
\eeqas
\end{theorem}
As with spherical surfaces, the case that the curve is a \emph{geodesic} is of course given by $N = {\bf n}$,
and then $\mu$, $\kappa_n$ and $\kappa_g$ are replace respectively by $\tau$, $\kappa$ and $0$ in the formula
for the boundary potential.    We remark that there is no \emph{singular} version of the geometric Cauchy problem
for CMC surfaces because these surfaces do not have any non-degenerate singularities. The only singularities that 
can occur are branch points, i.e., isolated points where $\dd f = 0$ (see, e.g. \cite{DH97}).

	\begin{figure}[ht]
\centering
$
\begin{array}{ccc}
\includegraphics[height=30mm]{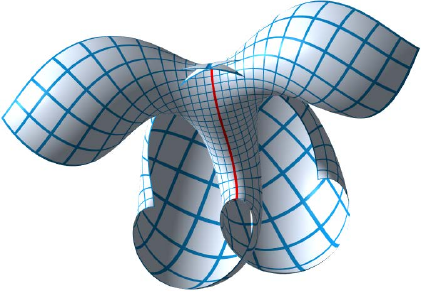}  &
\includegraphics[height=30mm]{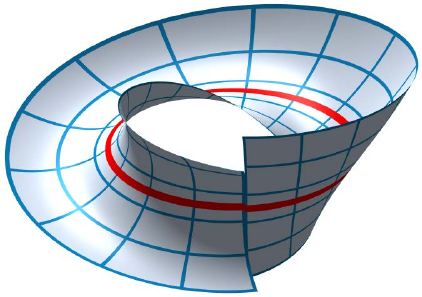}  &
\includegraphics[height=30mm]{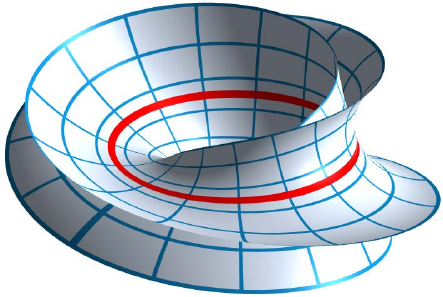}  \\
(\kappa_n(s),\kappa_g(s), \mu(s)) =(s,0,0) &  
\hbox{$x$ interval }  [0,2\pi]  &
\hbox{$x$ interval }  [0,4\pi]  
\end{array}
$
\caption{CMC surfaces in Example \ref{cmcexample1}.}
\label{figure6}
\end{figure}

\begin{example} \label{cmcexample1}
Figure \ref{figure6} shows, on the left, the unique solution corresponding to the inflectional geodesic curve with 
$\kappa(s)=s$ and $\tau(s)=0$. Unlike the spherical case (Example \ref{sphericalexample2}), the surface is regular
around the curve.   In the two images on the right, the solution for the same geometric Cauchy data as in the non-orientable
spherical surface of Example \ref{nonorientableexample}, is shown.  All non-minimal CMC surfaces are orientable, and
so this surface cannot (and does not) close up until the circle has been traversed twice.
\end{example}

\end{appendices}

\providecommand{\bysame}{\leavevmode\hbox to3em{\hrulefill}\thinspace}
\providecommand{\MR}{\relax\ifhmode\unskip\space\fi MR }
\providecommand{\MRhref}[2]{%
  \href{http://www.ams.org/mathscinet-getitem?mr=#1}{#2}
}
\providecommand{\href}[2]{#2}

\end{document}